\documentclass[11pt,reqno]{amsart}
\usepackage[margin=3.5cm]{geometry}
\usepackage[utf8]{inputenc}
\usepackage{amsmath, amssymb}
\usepackage{amsfonts}
\usepackage{amsthm}
\usepackage{mathtools}
\usepackage{centernot}
\usepackage{enumitem}
\usepackage{mathrsfs}
\usepackage{datetime}
\usepackage[hidelinks]{hyperref}
\usepackage{bookmark}

\newtheorem{theorem}{Theorem}[section]
\newtheorem{proposition}[theorem]{Proposition}
\newtheorem{lemma}[theorem]{Lemma}
\newtheorem{corollary}[theorem]{Corollary}
\newtheorem{question}[theorem]{Question}
\theoremstyle{definition}
\newtheorem{definition}[theorem]{Definition}
\newtheorem{example}[theorem]{Example}
\newtheorem{remark}[theorem]{Remark}

\numberwithin{equation}{section}
\DeclareEmphSequence{\bfseries\itshape}

\DeclareMathOperator{\cl}{cl}
\DeclareMathOperator{\fs}{fs}

\DeclareMathOperator{\gal}{Gal}

\DeclareMathOperator{\SL}{SL}

\DeclareMathOperator{\GL}{GL}
\DeclareMathOperator{\Span}{Span}
\DeclareMathOperator{\supp}{supp}

\newcommand{\m}{\mathfrak{m}}
\newcommand{\M}{\mathcal{M}}
\newcommand{\Q}{\mathbb{Q}}
\newcommand{\R}{\mathbb{R}}
\newcommand{\C}{\mathbb{C}}
\newcommand{\Z}{\mathbb{Z}}
\newcommand{\T}{\mathbb{T}}
\newcommand{\N}{\mathbb{N}}

\newcommand{\PP}{\mathcal{P}}

\newcommand\restr[2]{{
  \left.\kern-\nulldelimiterspace
  #1
  \right|_{#2}}}

\begin{document}

\title[Density of finitely supported invariant measures]{Density of finitely supported invariant measures for automorphisms of compact abelian groups}
\author{Rotem Yaari}
\address{Rotem Yaari, Department of Pure Mathematics, Tel Aviv University, Israel}
\email{rotemnaory@mail.tau.ac.il}

\begin{abstract}
We study the structure of invariant measures for continuous automorphisms of compact metrizable abelian groups satisfying the descending chain condition. We show that the finitely supported invariant measures are weak-* dense in the space of all invariant probability measures, and that if the system is ergodic with respect to Haar measure, then the finitely supported ergodic invariant measures are also dense. A key ingredient in the proof is a variant of the (periodic) specification property, which we establish for the ergodic systems in this class.
Our results also yield the following two consequences: first, that every finitely generated group of the form $\Z\ltimes G$, where $G$ is a countable abelian group, is Hilbert-Schmidt stable; and second, a Liv{\v s}ic-type theorem characterizing the uniform closure of coboundaries arising from continuous functions in terms of vanishing on periodic orbits.
We also construct an example showing that, in general dynamical systems, the property of having dense finitely supported invariant measures does not pass to product systems.
\end{abstract}



\maketitle

\section{Introduction}
\subsection*{Dense periodic measures and specification}
By a \emph{dynamical system} we mean a pair $(X,T)$ where $X$ is a compact metrizable space and $T:X\to X$ a homeomorphism.
In this paper we are primarily interested in dynamical systems consisting of a compact metrizable abelian group equipped with a continuous (group) automorphism. Such systems possess a richer structure than general dynamical systems, while still exhibiting interesting dynamical behavior.
For brevity, we will refer to such a dynamical system as an \emph{abelian group dynamical system}. 
We say that $x\in X$ is of \emph{period} $n\in \N$ if $T^n x = x$. Note that we do not require $n$ to be the least period of $x$, so that in particular $x$ is also of period $mn$ for every $m\in \N$. $x\in X$ is called \emph{periodic} if it is of some period $n\in\N$. Here and throughout the paper, we adopt the convention $\N=\{1,2,\dots\}$ and $\N_0=\N\cup\{0\}$.
The behavior of periodic points often captures key features of a dynamical system, such as its chaotic nature and entropy, thereby making their study a central problem in dynamics.
A basic question one can ask is whether they are dense in $X$.
However, even in the case that $(X,T)$ is an abelian group dynamical system, there are examples where the only periodic point in $X$ is its identity element \cite{aoki}, \cite[Examples 5.6]{schmidt_book}. It is therefore necessary to impose an additional assumption on the system, and a natural and useful one in this context is the \emph{descending chain condition} (\emph{dcc}); see Definition \ref{def dcc} below. In abelian group dynamical systems that satisfy the dcc, the periodic points are known to be dense \cite{laxton, kitchens} (see also \cite[Theorem 5.7]{schmidt_book}).

A common theme in dynamics is to adapt such questions from the point-set setting to the framework of invariant measures on the system\footnote{Most notably, Furstenberg’s well-known $\times 2,\times 3$ conjecture is the measure-theoretic analogue of his result stating that the only closed infinite 
$\times 2,\times 3$-invariant subset of the torus $\T$ is $\T$ itself \cite{furstenberg}.}.
For a dynamical system $(X,T)$, $\M(X)$ will denote throughout the paper the collection of Borel probability
measures on $X$ and $\M_T(X)$ will denote the subset of $\M(X)$ that consists of the $T$-invariant measures. We endow $\M(X)$ with the weak-* topology, which turns it into a metrizable compact space, and turns $\M_T(X)$ into a closed subspace.
Following the terminology of \cite{levit_vigdorovich}, we will say that $(X,T)$ has \emph{dense periodic measures} if the finitely supported measures in $\M_T(X)$ are dense in $\M_T(X)$. We will also say that $(X,T)$ has \emph{dense ergodic periodic measures} if the finitely supported \textit{ergodic} measures are dense in $\M_T(X)$.
We can now ask under what conditions a dynamical system has dense (ergodic) periodic measures.
If there exists a measure in $\mathcal{M}_T(X)$ whose support is all of $X$, then a necessary condition is that the periodic points are dense. This applies to abelian group dynamical systems, as the Haar measure on $X$ is always $T$-invariant. However, there exist abelian group dynamical systems with dense periodic points that do not admit dense periodic measures \cite{eckhardt}.
These questions have attracted the attention of many authors, who have provided affirmative answers in various settings, see \cite{sigmund, coudene,nonuniform, weak_and_entropy, on_density} to name a few.
In the setting of abelian group dynamical systems, it follows from the work of Parthasarathy that Bernoulli shifts over groups have dense ergodic periodic measures \cite{Parthasarathy}, and Levit and Vigdorovich showed that more general group shift actions also admit dense periodic measures \cite[Proposition 8.6]{levit_vigdorovich} (with the obvious adaptation of the definition to actions of groups more general than $\Z$). 
Hyperbolic toral automorphisms were shown to have dense ergodic periodic measures by Lind \cite{lind_spec}. This result was later generalized by Marcus \cite{marcus} to toral automorphisms that are ergodic with respect to Haar measure, and to hyperbolic automorphisms of solenoids satisfying the dcc in \cite{solenoidal}.

Our first main result completes the picture for the class of abelian group dynamical systems that satisfy the dcc, answering a question posed by Levit and Vigdorovich \cite[Question 2]{levit_vigdorovich} and also one of Eckhardt \cite[p.\ 2]{eckhardt}.
\begin{theorem}\label{thm dpm intro}
    Let $(X,T)$ be an abelian group dynamical system that satisfies the dcc. Then $(X,T)$ has dense periodic measures. Moreover, if $T$ is ergodic with respect to the Haar measure on $X$ then $(X,T)$ has dense ergodic periodic measures.
\end{theorem}
In particular, every automorphism of $\T^d$ has dense periodic measures. We note that this result is new even in the case where $T$ is an automorphism of $\T^d$ induced by a block-diagonal integer matrix with hyperbolic and unipotent blocks. While the result is known separately for hyperbolic and unipotent toral automorphisms, the challenge of extending dense periodic measures to products, discussed later, makes this case nontrivial.

To prove the second part of the theorem, we establish a specification-like result. The specification property, originally introduced by Bowen in his study of Axiom A diffeomorphisms \cite{bowen_spec}, is a powerful dynamical tool with numerous applications.
Roughly speaking, specification means that orbit segments of
multiple points can be approximated by a single, sometimes periodic, orbit. The specification
property has been extensively studied ever since. Many dynamical systems have been shown to
exhibit weaker forms of specification, which are nevertheless strong enough to yield interesting consequences; we refer to the comprehensive survey \cite{panorama} for an overview of the variations and applications of specification.
Similarly, the variant we introduce, called \emph{partial specification} (Definition \ref{def partial specification}), is strong enough to produce most, if not all, of the interesting consequences of specification, making the following result notable in its own right.
\begin{theorem}\label{thm spec intro}
    Let $(X,T)$ be an abelian group dynamical system, ergodic with respect to the Haar measure on $X$ and satisfying the dcc. Then $(X,T)$ satisfies partial specification.
\end{theorem}
A related result was proved by Dateyama \cite{dateyama}, who showed that an abelian group dynamical system is ergodic with respect to the Haar measure if and only if it satisfies the property now known as weak specification. The key difference is that partial specification yields a \emph{periodic} tracing point, whereas Dateyama's weak specification does not. Consequently, Dateyama's theorem cannot be used to deduce Theorem \ref{thm dpm intro}. Moreover, Dateyama's result does not require the dcc, and, as mentioned earlier, Theorem \ref{thm dpm intro} fails in this greater generality \cite{aoki}, \cite[Examples 5.6]{schmidt_book}.

One difficulty in working with dense periodic measures is that it is unclear whether this property persists under group extensions. This issue persists even in the following simpler case: given two abelian group dynamical systems with dense periodic measures, does their product admit dense periodic measures? Outside the realm of groups, we demonstrate that this is not true in general:
\begin{theorem}\label{thm product intro}
    There exist two dynamical systems with dense periodic measures whose product does not admit dense periodic measures.
\end{theorem}
Much of the proof presented in this paper is designed to circumvent this issue by using the stronger property of partial specification which behaves well under extensions (unlike some other specification-like properties \cite{lind_spec}).
\subsection*{Hilbert-Schmidt stability}
Much of the motivation for the present work comes from the stability theory of groups. Let $G$ be a countable group. Roughly speaking, $G$ is called \emph{Hilbert-Schmidt stable} if every almost homomorphism from $G$ into the unitary group lies near a genuine homomorphism from $G$ into the unitary group (see \cite[Section 9]{levit_vigdorovich} or
any other reference on the topic for a precise definition). The study of Hilbert–Schmidt stability is deeply intertwined with questions about hyperlinear groups and the structure of operator algebras \cite{hyperlinear}.
Among amenable groups, Hilbert-Schmidt stability was established for abelian groups by Glebsky \cite{glebsky2010}, and later extended to virtually abelian groups and the discrete Heisenberg group by Hadwin and Shulman \cite{hadmin_shulman1}. Levit and Vigdorovich \cite{levit_vigdorovich} proved this property for all finitely generated virtually nilpotent groups, as well as for many other metabelian groups. For finitely generated nilpotent groups, it was independently established by Eckhardt and Shulman \cite{eckhardt_shulman_amenable}.
Quantitative aspects of Hilbert-Schmidt stability for abelian and amenable groups were studied in \cite{glebsky2010} and \cite{characters}, respectively.

The following open question appears in \cite{levit_vigdorovich}.
\begin{question}[{\cite[Question 5]{levit_vigdorovich}}]\label{question metabelian}
    Are all finitely generated metabelian groups Hilbert-Schmidt stable? Equivalently, is every finitely generated group of the form $\Z^n\ltimes G$, where $n\in\N$ and $G$ is a countable abelian group, Hilbert-Schmidt stable?
\end{question}
Both works \cite{levit_vigdorovich, eckhardt_shulman_amenable} highlighted a connection between the Hilbert-Schmidt stability of metabelian groups and the density of periodic measures in abelian group dynamical systems, obtained by dualizing normal abelian subgroups of the original group. In light of their observation, the next result, which provides a step forward in addressing Question \ref{question metabelian}, follows almost immediately from Theorem \ref{thm dpm intro}.
\begin{corollary}\label{cor hs}
    Every finitely generated group of the form $\Z\ltimes G$, where $G$ is a countable abelian group, is Hilbert-Schmidt stable. 
\end{corollary}
We note that counterexamples to Question \ref{question metabelian} exist if the finite generation assumption is removed \cite{levit_vigdorovich, eckhardt}.
Moreover, there are infinitely many finitely generated $4$-step solvable groups that are not Hilbert-Schmidt stable \cite{eckhardt}. 
\subsection*{Liv{\v s}ic theory} 
Let $(X,T)$ be a dynamical system.
A function $f :X \to \R$ is called a \emph{coboundary} if there exists another function $P:X\to \R$ such that $f = P\circ T- P$.
It is clear that every coboundary $f$ sums to $0$ on every periodic orbit, namely, if $T^n x=x$ for some $x\in X$ and $n\in \N$, then
\begin{equation}\label{eq periodic sum}
    \sum_{i=0}^{n-1} f(T^i x) = 0.
\end{equation}
In his two papers \cite{livsic1,livsic2}, Liv{\v s}ic initiated the study of conditions under which the converse holds.
He showed that for certain dynamical systems and classes of functions, a function $f$ that satisfies Equation \eqref{eq periodic sum} for all periodic points is a coboundary, and moreover, that the regularity of the corresponding function $P$ depends on that of $T$ and $f$. Liv{\v s}ic's results have been extensively studied and generalized to broader contexts, see \cite{katok_construct, Llave} and references therein.
Due to the connection between this question and the density of periodic measures, as observed by various authors including \cite{veech, coudene}, we have the following result.
\begin{corollary}\label{cor livshitz}
    Let $(X,T)$ be an abelian group dynamical system satisfying the dcc. Then $f\in C(X)$ satisfies Equation \eqref{eq periodic sum} for every periodic point $x\in X$ if and only if $f$ belongs to the closure of $\{P\circ T - P : P\in C(X)\}$ (where the closure is taken in the uniform norm of $C(X)$).
\end{corollary}
Note that, while Liv{\v s}ic-type theorems are typically formulated in the setting of manifolds, where the notions of smoothness and regularity of functions are well-defined, there are dynamical analogs of these concepts.
If $(X,T)$ is an abelian group dynamical system, ergodic with respect to Haar measure, then it should be possible to use the partial specification (replacing the regularity assumption on $T$ in manifold settings) to derive a Liv{\v s}ic-type theorem that does not require taking a closure as in the previous corollary. As with other applications of partial specification (rather than of the dense periodic measures), we omit the details and refer to \cite[Section 11.3]{katok_construct}.
\subsection*{Outline and organization of the paper}
In Section \ref{section pre}, we review key facts and establish the notation used throughout the paper.
In Section \ref{Section automorphisms structure}, we present a variant of the structure theorem for abelian group dynamical systems satisfying the dcc. This theorem shows that such systems can be expressed as a finite chain of extensions of virtually unipotent toral automorphisms, finite groups, Bernoulli shifts and ergodic automorphisms of solenoids. We also briefly review the theory of algebraic actions, which will be used in the proof presented in this section.
Section \ref{section spec} introduces the concept of partial specification and proves its preservation under group extensions.
In Section \ref{section solenoid}, we begin by demonstrating that the solenoids discussed earlier satisfy partial specification. We then combine all prior results to address abelian group dynamical systems ergodic with respect to Haar measure, leading to the proof of Theorem \ref{thm spec intro} and the ergodic part of Theorem \ref{thm dpm intro}.
In Section \ref{section unipotent}, we investigate ergodic measures associated with unipotent toral automorphisms and prove that these automorphisms exhibit a strong form of dense periodic measures. This result is then combined with the partial specification result obtained earlier in Section \ref{section dpm}, culminating in the proof of Theorem \ref{thm dpm intro}.
In Section \ref{section other main proofs} we prove Corollaries \ref{cor hs} and \ref{cor livshitz}. Finally,
Section \ref{section product} presents the proof of Theorem \ref{thm product intro}, which does not rely on the preceding results.

\bigskip

\textbf{Acknowledgements.} I am grateful to Arie Levit for many valuable ideas and suggestions that contributed to the development of this paper, and to Itamar Vigdorovich for fruitful discussions and comments.
I thank Barak Weiss for encouraging the broader scope of this work. I also thank the anonymous referee for a careful reading of the manuscript and for numerous helpful comments and suggestions.

The author was supported by ISF grant 1788/22.

\textbf{AI disclosure.} ChatGPT was used solely for stylistic editing. All mathematical work in the paper is the author’s own.

\bigskip

\section{Preliminaries}\label{section pre}
Recall that in this paper, a \emph{dynamical system}, or simply a \emph{system}, refers to a pair $(X,T)$ where $X$ is a compact metrizable space and $T:X\to X$ a homeomorphism. A dynamical system $(X',T')$ is said to be a \emph{continuous factor} of the system $(X,T)$ if there exists a continuous surjection, called a \emph{continuous factor map} $\varphi: X \to X'$ such that $\varphi\circ T=T'\circ \varphi$. If in addition $\varphi$ is a bijection then $(X,T)$ and $(X',T')$ are said to be \emph{conjugate}.

Now let $(X,T)$ be an abelian group dynamical system, i.e., a compact metrizable abelian group together with a continuous automorphism (which then must be a homeomorphism). We use additive notation for $X$ and denote its identity element by $0$.
Let $\widehat{X}$ denote the Pontryagin dual of $X$.
If $X'$ is another compact abelian group and $\varphi :X\to X'$ is a continuous homomorphism, then the \emph{dual homomorphism} $\widehat{\varphi}:\widehat{X'}\to \widehat{X}$ is determined by the relation $\langle \varphi(x),g\rangle = \langle x, \widehat{\varphi}(g)\rangle$ for all $x\in X$ and $g\in \widehat{X'}$. In particular, $\widehat{T}$ is an automorphism of $\widehat{X}$.

Two abelian group dynamical systems $(X,T)$ and $(X',T')$ are said to be \emph{isomorphic} if there exists a continuous group isomorphism $\varphi:X\to X'$ such that $\varphi\circ T=T'\circ \varphi$.

A subgroup $Y$ of $X$ is called \emph{$T$-invariant} if $TY=Y$.
Let $Y$ be a $T$-invariant closed subgroup of $X$. To simplify the notation, we will continue to denote by $T$ both its restriction to $Y$ and the induced automorphism on $X/Y$, provided there is no risk of confusion. 
We will always choose an $X$-invariant metric $d_X$ on $X$ (compatible with the topology); such a metric is known to exist.
We can then define the metric $d_Y$ on $Y$ to be the restriction of the metric on $X$, and
    \begin{equation}\label{eq metric on quotient}
        d_{X/Y} (x_1 +Y, x_2 + Y) = \min_{y\in Y} d_X(x_1,x_2 + y).
    \end{equation}
It is well known that $d_{X/Y}$ is a translation-invariant metric on $X/Y$.

An abelian group dynamical systems is called \emph{ergodic} if it is ergodic with respect to the Haar measure on the group.

We will primarily focus on abelian group dynamical systems satisfying the descending chain condition.
\begin{definition}\label{def dcc}
    $(X,T)$ satisfies the \emph{descending chain condition} (\emph{dcc}) if for every chain $X\ge X_1 \ge X_2 \ge\cdots$ of non-increasing $T$-invariant closed subgroups of $X$ there exists an integer $K$ such that $X_k = X_K$ for all $k\ge K$.
\end{definition}
Note that if $(X,T)$ satisfies the dcc and $Y$ is a $T$-invariant closed subgroup of $X$, then both $(Y,T)$ and $(X/Y,T)$ also satisfy the dcc. The reader is referred to \cite{schmidt_book} for more details\footnote{We mention that the dcc already ensures that $X$ is metrizable \cite[Proposition 5.1]{jaworski}, so the metrizability assumption can in fact be omitted in this case.}.
\subsection{Solenoids} \label{subsection solenloids notation}
Solenoids have been studied in dynamical systems through various models \cite{berend, rigidity_properties, lawton, lind_skew, schmidt_book, solenoidal}. We now introduce the $S$-adic solenoid, a model for a collection of solenoids that is particularly convenient for our purposes\footnote{Although this construction is sufficient for our purposes, we highlight that solenoids like $\widehat{\Q}$ cannot be realized in this manner. A more general construction can be found in \cite{rigidity_properties}. This restriction is vital, as our main results fail in this broader context \cite[Examples 5.6]{schmidt_book}.}.
Let $S$ be a set that consists of $\infty$ together with a finite (possibly empty) collection of prime numbers.
We denote by $\Z[\frac{1}{S}]$ the ring of rational numbers whose denominators are divisible only by primes in $S\setminus\{\infty\}$.
As usual, for a prime $p$, $\Q_p$ denotes the field of $p$-adic numbers and we set $\Q_\infty = \R$.
For $d\in \N$, $\Q^d$ embeds diagonally in $\Q_S^d\coloneqq\prod_{p\in S} \Q_p^d$, and denote by $\Delta_{S}^d$ the image of $\Z[\frac{1}{S}]^d$, namely\footnote{It should be noted that $\Q_S^d$ and $\Delta_S^d$ are not the $d$-fold products of $\Q_S^1$ and $\Delta_S^1$, respectively, but are isomorphic to them.},
\begin{equation*}
    \Delta_{S}^d = \{(x,\dots,x):x\in \Z[\frac{1}{S}]^d\}\subseteq \Q_S^d.
\end{equation*}
Then $\Delta_{S}^d$ is a cocompact discrete subgroup of $\Q_S^d$, yielding the \emph{$d$-dimensional $S$-adic solenoid} $X_{S}^d = \Q_S^d/\Delta_{S}^d$.
Recall that $\widehat{\Q_S^d}$ can be identified with $\Q_S^d$, and that under this identification,
    \begin{equation}\label{eq solenoid dual}
       \widehat{X_{S}^d} = \{(-x,x,\dots,x):x\in \Z[\frac{1}{S}]^d\}\subseteq\R^d\times\prod_{p\in S\setminus\{\infty\}}\Q_p^d \simeq \widehat{\Q_S^d}
    \end{equation}
(see \cite[Chapter 3, Section 1.6]{representation}). 
Note that for a ring $R\subseteq \Q$, the group $\GL_d(R)$ consists of those matrices $A\in \GL_d(\Q)$ for which both $A$ and $A^{-1}$ have all entries in $R$.
A matrix $A\in \GL_d(\Z[\frac{1}{S}])$
defines an automorphism of $X_{S}^d$ by acting on it diagonally, and its dual automorphism $\widehat{A}:\widehat{\Delta_{S}^d} \to \widehat{\Delta_{S}^d}$ is given by the diagonal action of the transpose $A^\top$.
\section{Automorphisms of compact abelian groups}\label{Section automorphisms structure}
In this section, we present a variation of Miles and Thomas’s structure theorem for automorphisms of compact groups \cite{breakdown} (see also \cite[Theorem 7.1]{jaworski}). It is divided into two parts: Propositions \ref{prop ergodic structure} and \ref{prop nonergodic structure}, addressing the ergodic and non-ergodic cases respectively. We begin with the former.
\begin{proposition}\label{prop ergodic structure}
    Let $(X,T)$ be an ergodic abelian group dynamical system satisfying the dcc. Then $(X,T)$ is a continuous factor of an ergodic abelian group dynamical system $(X',T')$ which admits a finite chain of $T'$-invariant closed subgroups $0=X_n\le \cdots \le X_0 = X'$ such that for every $k = 0 ,\dots,n-1$, $(X_k/X_{k+1}, T')$ is ergodic and is either\footnote{In fact, by \cite[Corollary 6.3]{schmidt_book}, more can be said about the order in which these factors appear, but we will not make use of this.}:
    \begin{enumerate}
        \item isomorphic to a Bernoulli shift $(Y^\Z,\sigma)$ where $Y$ is either a finite abelian group or $Y=\T$ (and $\sigma$ is the left shift, defined by $\sigma(y_n)_{n\in\Z}=(y_{n+1})_{n\in\Z}$);
        \item \label{prop ergodic structure, solenoid} a continuous factor of an $S$-adic solenoid equipped with an automorphism induced by a matrix whose characteristic polynomial is irreducible over $\Q$ and has no roots of unity among its roots.
    \end{enumerate}
\end{proposition}
The proof relies on well known facts from the theory of algebraic $\Z^d$-actions \cite{kitchens, schmidt_book}, which we now recall, restricting attention to the case of a $\Z$-action relevant to our work.
Let $R_1 = \Z[u,u^{-1}]$ denote the ring of Laurent polynomials in the variable $u$ and coefficients in $\Z$.
Given an abelian group dynamical system $(X,T)$, we define an action of $R_1$ on $\widehat{X}$ as follows: for $f = \sum_{n\in\Z}c_nu^n\in R_1$ (in particular, $c_n\ne0$ for only finitely many $n\in\Z$) and $g\in \widehat{X}$,
\[f \cdot g = \sum_{n\in \Z} c_n\widehat{T}^n g.\]
This action turns $\widehat{X}$ into an $R_1$-module called the \emph{dual module} of $(X,T)$.
On the other hand, let $M$ be a countable $R_1$-module. Thinking of $M$ as a discrete abelian group, the action of $u\in R_1$ defines an automorphism of $M$, and we define the \emph{dual dynamical system of $M$} to be $(\widehat{M},\widehat{u})$.
There is a deep and rich interplay between the algebraic structure of this module and the dynamical properties of the system $(X,T)$.
\begin{proof}[Proof of Proposition \ref{prop ergodic structure}]
    By Corollary 6.3 and Theorem 6.5(1) in \cite{schmidt_book}, $(X,T)$ is a continuous factor of an ergodic abelian group dynamical system $(X',T')$ which admits a finite chain of $T'$-invariant closed subgroups $0=X_n\le \cdots \le X_0 = X'$ with the following property: for each $k$, there exists a prime ideal $\mathfrak{p}_k\subseteq R_1$ such that the system $(X_k/X_{k+1}, T')$ is ergodic and isomorphic to the dual system of the $R_1$-module $R_1/\mathfrak{p}_k$. 
    If $\mathfrak{p}_k$ is a non-principal ideal, then using the fact that $\Q[u,u^{-1}]$ is a principal ideal domain it can be shown that $R_1/\mathfrak{p}_k$ is finite \cite[Examples 6.17(3)]{schmidt_book}, and hence so is its dual $X_k/X_{k+1}$. However, a finite group with an ergodic automorphism must be trivial, in which case the subgroup $X_{k+1}$ can be omitted. 
    Otherwise, $\mathfrak{p}_k$ is generated by some irreducible element $f\in R_1$, and since $u\in R_1$ is a unit, we may assume that $f$ is either $0$ or of the form
    \begin{equation}\label{eq f form}
        f = \sum_{n=0}^{d}c_nu^n\in \Z[u],\qquad c_0\ne 0.
    \end{equation}
    In this case, the dual system is given by
    \[(R_1/\mathfrak{p}_k)\widehat{\phantom{ll}} =\{(x_n)\in \T^\Z: \sum_{n=0}^d c_nx_{m+n} = 0 \text{ for every }m\in\Z\}\]
    together with the left shift (see \cite[Examples 5.2]{schmidt_book}).
    If $f=0$ then we obtain the Bernoulli shift on $\T^\Z$. If $f=c\in \Z$ is a nonzero constant polynomial, then $(x_n)\in (R_1/\mathfrak{p}_k)\widehat{\phantom{ll}}$ if and only if $cx_n=0$ for all $n\in\Z$, and hence $(R_1/\mathfrak{p}_k)\widehat{\phantom{ll}}$ is isomorphic to the Bernoulli shift on $(\Z/c\Z)^\Z$ (once again, the case $c=1$ which gives the trivial group can be omitted).

    Finally, we consider the case where $f$ is non-constant. Let $\alpha$ be any root of $f$ and consider the ideal $\mathfrak{p} = \{h\in R_1: h(\alpha)=0\}\subseteq R_1$.
    Since the constant polynomials embed into $R_1/\mathfrak{p}$, the latter is infinite, and it follows once again that $\mathfrak{p}$ is a principal ideal. Since $f\in \mathfrak{p}$ and is irreducible, we conclude that $\mathfrak{p}=(f)$.

    We define a map $\varphi: R_1/(f)\to \Q^d$ as follows. First, we embed $R_1/(f)$ in the field $\Q(\alpha)$ by sending each element $h+(f)$ to $h(\alpha)\in \Q(\alpha)$ (cf.\ \cite[Section 7]{schmidt_book}). Next, we identify $\Q(\alpha)$ with $\Q^d$ by mapping the basis $1,\alpha,\dots,\alpha^{d-1}$ (over $\Q$) to the standard basis of $\Q^d$. We define $\varphi$ to be the composition of these two maps. Then $\varphi$ is injective. We equip $\varphi(R_1/(f))$ with the $R_1$-module structure induced via $\varphi$, that is, $r \varphi(h) \coloneqq \varphi(rh)$ for every $r\in R_1$ and $h\in R_1/(f)$.
    Explicitly, $u\in R_1$ acts via the companion matrix of $f/c_d$, denoted $A$, which is invertible since $f(0)\ne 0$ by Equation \eqref{eq f form}.  
    Let $S$ be the set of all primes dividing the denominators of the entries of $A$ and $A^{-1}$, together with $\infty$, so that $A\in \GL_d(\Z[\frac{1}{S}])$.
    Observe in addition that the module $\varphi(R_1/(f))$ is generated by the element $\varphi(1+(f))$, which is the first standard basis vector in $\Q^d$.
    It follows that $\varphi(R_1/(f))$ is a submodule of $\Z[\frac{1}{S}]^d$ (where the action of $u\in R_1$ is again given by $A$).
    
    We have embedded $R_1/(f) = R_1/\mathfrak{p}_k$ in $\Z[\frac{1}{S}]^d$, which is identified with $\widehat{X_{S}^d}$ as in \eqref{eq solenoid dual}. The dual of this embedding is a continuous factor map from $(X_{S}^d, A^\top)$ onto $((R_1/\mathfrak{p}_k)\widehat{\phantom{ll}},\widehat{u})$.
    
    We conclude by noting that $A^\top$ has the required properties.
    Equation \eqref{eq f form} and the irreducibility of $f$ in $R_1$ imply together that $f$ is also irreducible in $\Z[u]$, and hence over $\Q$.
    Moreover, the ergodicity of $(R_1/(f))\widehat{\phantom{ll}}$ is equivalent to $f$ having no roots of unity among its roots \cite[Theorem 6.5(1)]{schmidt_book}.
\end{proof}
We could not find the following simple lemma in the literature.
\begin{lemma}\label{lemma ergodic subgroup and quotient}
    Let $(X,T)$ be an abelian group dynamical system.
    Let $Y$ be a $T$-invariant closed subgroup of $X$, and suppose that $(Y,T)$ and $(X/Y,T)$ are ergodic. Then $(X,T)$ is ergodic as well.
\end{lemma}
\begin{proof}
    An abelian group dynamical system $(X,T)$ is ergodic if and only if the only $\widehat{T}$-periodic character is the trivial one \cite{halmos}, \cite[Remarks 1.7(3)]{schmidt_book}.
    In other words, given $\chi\in \widehat{X}$ such that $\widehat{T}^n\chi=\chi$ for some $n\in \N$, we wish to prove that $\chi$ is trivial. The restriction of $\chi$ to $Y$ gives a $\widehat{T}$-periodic character on $Y$, which must be trivial by the ergodicity of $(Y,T)$. Thus, $\chi$ descends to a $\widehat{T}$-periodic character on $X/Y$, which once again must be trivial, and we conclude that $\chi$ is trivial.
\end{proof}
The second main result of this section is as follows.
\begin{proposition}\label{prop nonergodic structure}
    Let $(X,T)$ be an abelian group dynamical system satisfying the dcc. Then there exist $T$-invariant closed subgroups $X_2\le X_1 \le X$ such that
    \begin{enumerate}
        \item $X/X_1$ is finite.
        \item $(X_1/X_2, T)$ is isomorphic to a torus equipped with an automorphism induced by a matrix all of whose eigenvalues are roots of unity.
        \item $(X_2,T)$ is ergodic.
    \end{enumerate}
\end{proposition}
\begin{proof}
    By \cite[Proposition 3.5]{schmidt_book} there exists a maximal $T$-invariant closed subgroup $X_2$ such that $(X_2,T)$ is ergodic, and moreover, $Y\coloneqq X/X_2$ is a compact Lie group. Denote by $\pi :X\to Y$ the canonical projection and by $Y^\circ$ the connected component of the identity element of $Y$. Set $X_1 = \pi^{-1} (Y^\circ)$, then it is $T$-invariant and $X/X_1 \simeq Y / Y^\circ$ which is finite since $Y$ is compact.

    Next, since $X_1/X_2 = \pi(X_1) = Y^\circ$ is a connected compact abelian Lie group, we may, without loss of generality, assume that $X_1/X_2=\T^d$ and that $T\in \GL_d(\Z)$ (as an automorphism of $X_1/X_2)$. Suppose, for contradiction, that $T$ has an eigenvalue that is not a root of unity. Then its characteristic polynomial must have a factor that is irreducible over $\Q$ and not cyclotomic.
    By the primary decomposition theorem (see \cite[Theorem 12]{kunze}), there exists a rational $T$-invariant subspace of $\R^d$ on which all eigenvalues of $T$ are not roots of unity.
    Let $Y_1$ denote the projection of this subspace to $\T^d$. Then $Y_1$ is a $T$-invariant closed subgroup of $Y=X/X_2$ on which $T$ acts ergodically. Applying Lemma \ref{lemma ergodic subgroup and quotient} we conclude that $T$ acts ergodically on the $T$-invariant closed subgroup $\pi^{-1}(Y_1)$ of $X$. Since this subgroup strictly contains $X_2$, this contradicts the maximality of $X_2$.
\end{proof}
\section{Partial specification}\label{section spec}
In this section, we introduce partial specification, a variant of the specification property, and show that it is preserved under extensions of abelian group dynamical systems (Proposition \ref{prop extension has partial spec}). For background on the specification property, see \cite{panorama}.
\begin{definition}[cf.\ {\cite[Definitions 1,2]{panorama}}]
\label{def spec and partial tracing}
    Let $(X,T)$ be a dynamical system and let $M\in\N$. Let $a_1 < b_1 < \cdots < a_r < b_r$ be some non-negative integers and $x_1,\dots, x_r\in X$. The collection $\xi = \{(x_i; a_i,b_i)\}_{i=1}^r$ is called an \emph{$M$-spaced specification} if $a_{i+1} - b_i \ge M$ for every $i$.

    Fix a metric $d_X$ on $X$, and let $\varepsilon > 0$.
    We say that the specification $\xi$ is \emph{$\varepsilon$-partially traced by $y\in X$} (\emph{with respect to $T$}) if for every $i=1,\dots,r$, \begin{equation*}
        \frac{\lvert\Lambda_i\rvert}{b_i-a_i} > 1-\varepsilon \qquad\text{where}\qquad \Lambda_i = \{a_i\le n < b_i :d_X(T^n x_i, T^{n}y) < \varepsilon\}.
    \end{equation*}
    We will refer to $\Lambda_i$ as the collection of \emph{tracing indices}.
\end{definition}
\begin{definition}\label{def partial specification}
    Let $(X,T)$ be a dynamical system and let $\PP\subseteq \N$ be a subset with bounded gaps. $(X,T)$ satisfies \emph{partial specification with periods $\PP$} if for every $\varepsilon>0$ there exist $N,M\in\N$ such that
    for every $M$-spaced specification $\xi=\{(x_i; a_i,b_i)\}_{i=1}^r$ and every $n\in \PP$ with $n \ge \max\{N,(1+\varepsilon)b_r\}$, there exists $y\in X$ of period $n$ that $\varepsilon$-partially traces $\xi$.
        
    Moreover, we say that $(X,T)$ satisfies \emph{partial specification} if for every $c\in \N$, there exists $\PP\subseteq c\N$ such that $(X,T)$ satisfies partial specification with periods $\PP$.
\end{definition}
The main novelty of partial specification, compared to existing specification-like properties, is the strong control over the tracing point's period. This flexibility will play a crucial role in constructing periodic points in group extensions in Section \ref{section dpm}.

Aside from the period requirement, partial specification is stronger than weak specification \cite{marcus, dateyama, panorama}\footnote{Weak specification was introduced by Marcus \cite{marcus}, although it was not given that name there.} in that the gaps between consecutive orbit segments are uniformly bounded by a constant independent of the lengths of the segments. At the same time, it is weaker than specification and weak specification, since the tracing point is only required to approximate most of the points in each orbit segment rather than all of them. This feature is shared with the almost product property \cite{almost_prod} and the almost specification property \cite{almost_spec}. However, they require the allowed proportion of tracing indices to converge to $1$ as the orbit segment length tends to infinity, while allowing $1$-spaced specifications. Moreover, they involve several additional parameters that provide greater flexibility in controlling the tracing process, but are not relevant for our purposes.

We do not investigate the precise relationships between partial specification and the above specification-like properties; in particular, the corresponding implications remain unclear. After the first version of this paper appeared on arXiv, Can and Kulczycki \cite{can2026interplay} initiated the study of the relationships between partial specification and other dynamical properties.
\begin{remark}\label{remarks partial spec}
    \begin{enumerate}
        \item \label{remark independent of metric} If $d_X$ and $d_X'$ are two metrics on $X$, then for every $\varepsilon>0$ we can choose $0<\delta<\varepsilon$ such that $d_X(x,y)<\delta \Rightarrow d_X'(x,y)<\varepsilon$ for all $x,y\in X$. Using this observation, one easily verifies that whether $(X,T)$ satisfies partial specification (with periods $\PP$) is independent of the choice of metric.
        \item \label{remark partial spec ergodic} Given nonempty open subsets $U,V \subseteq X$, partial specification allows us to construct a point whose orbit intersects both $U$ and $V$. Therefore, $T^nU\cap V\ne\emptyset$ for some $n\in \Z$, implying that $(X,T)$ is topologically transitive \cite[Theorem 5.8]{walters}. In the case that $(X,T)$ is an abelian group dynamical system, this is equivalent to $(X,T)$ being ergodic (see \cite[Theorem 1.1]{schmidt_book}).
    \end{enumerate}
\end{remark}
We will make use of the following simple example later on.
\begin{example}\label{example bernoulli has spec}
    Let $X$ be a compact metrizable space and $T: X^\Z\to X^\Z$ the left shift. For any $\varepsilon>0$, we can take $M=1$, and every $M$-spaced specification can be $\varepsilon$-partially traced by a point of period $n$ for any $n> b_r$. Therefore we can choose the collection of periods to be any subset of $\N$ with bounded gaps, and in particular $(X^\Z,T)$ satisfies partial specification.
    In fact, this flexibility in choosing the parameters reflects the fact that $(X^\Z,T)$ satisfies the stronger notion of periodic specification (see e.g.\ \cite{sigmund}).
\end{example}
The following is the main result of this section.
\begin{proposition}\label{prop extension has partial spec}
    Let $(X,T)$ be an abelian group dynamical system and $Y$ a $T$-invariant closed subgroup of $X$. Suppose that $(Y,T)$ and $(X/Y,T)$ satisfy partial specification with the same collection of periods, then so does $(X,T)$.
\end{proposition}
To prove this, we introduce several auxiliary notions. The interplay between variants of specification and shadowing has been extensively studied; see, for example, \cite{various_shadowing,almost_spec_shadowing, panorama}. We now define a variant, which we call \emph{partial shadowing}, and which arises naturally from partial specification.
Additionally, we define the notion of \emph{partial tracing} of a sequence, which resembles the partial tracing introduced earlier for a specification.
\begin{definition}\label{def partial shadowing}
    Let $(X,T)$ be a dynamical system, $d_X$ a metric on $X$ and $(x_n)_{n=1}^N$ a finite sequence in $X$.
    Given $\varepsilon>0$, we say that $(x_n)_{n=1}^N$ is \emph{$\varepsilon$-partially traced by $y\in X$} (\emph{with respect to $T$}) if
    \[\frac{\lvert \Lambda\rvert}{N} > 1- \varepsilon \qquad\text{where}\qquad \Lambda = \{1\le n\le N:d_X(x_n,T^n y)<\varepsilon\}.\]
    Once again, we will refer to $\Lambda$ as the collection of \emph{tracing indices}.
    Let $\delta>0$. $(x_n)_{n=1}^N$ will be called a \emph{$\delta$-partial pseudo-orbit} (\emph{for $T$}) if either $N=1$ or $N\ge 2$ and
    \[\frac{1}{N-1}{\Big\lvert} \{1\le n\le N-1:d_X(Tx_n,x_{n+1})<\delta\}{\Big\rvert} > 1-\delta.\]
    We say that $(X,T)$ satisfies \emph{partial shadowing}
    if for every $\varepsilon>0$ there exists $\delta>0$ such that every $\delta$-partial pseudo-orbit is $\varepsilon$-partially traced by some point in $X$.
\end{definition}
Partial shadowing is closely related to several shadowing variants in the literature, particularly the $M_\alpha$-shadowing property of \cite{various_shadowing} (see also \cite{subshadow}). The main difference is that those properties concern infinite pseudo-orbits, whereas partial shadowing imposes analogous conditions on finite pseudo-orbits. We do not investigate the precise relationships between partial shadowing and these shadowing properties; see \cite{can2026interplay} for further results relating partial shadowing to other dynamical properties.

As in Remark \ref{remarks partial spec}(\ref{remark independent of metric}), whether a system satisfies partial shadowing is independent of the specific metric.
\begin{lemma}\label{lemma spec implies shadow}
    Suppose the dynamical system $(X,T)$ satisfies partial specification with periods $\PP\subseteq \N$. Then $(X,T)$ satisfies partial shadowing.
\end{lemma}
\begin{proof}
    Fix a metric $d_X$ on $X$ and let $\varepsilon>0$. The partial specification provides us with $M\in \N$ corresponding to $\varepsilon/4$ as in Definition \ref{def partial specification} (the number $N$ from that definition will not be used).
    Fix some $L\in \N$ that satisfies $\frac{M}{L} < \frac{\varepsilon}{4}$, and then fix a large integer $C>L$ such that $\frac{L+M}{C}<\frac{\varepsilon}{4}$.
    Since $T$ is uniformly continuous, we can choose $\delta>0$ such that for every $x,y\in X$,
    \begin{equation}\label{eq uniformly continuous bound on metric}
        d_X(x,y)<\delta \Longrightarrow d_X(T^n x, T^n y)<\frac{\varepsilon}{2C} \quad \text{for}\quad n=0,\dots,C.
    \end{equation}
    Moreover, we choose $\delta$ small enough so that $\delta<1/C$.
    Observe that Equation \eqref{eq uniformly continuous bound on metric} yields the following estimate, which will be used later on: if $y_0,\dots,y_C \in X$ satisfy $d_X(Ty_{n-1},y_{n}) < \delta$ for $n=1,\dots,C$, then for every such $n$,
    \begin{equation}\label{eq delta psuedo orbit estimate}
        d_X(T^{n}y_0, y_{n}) \le \sum_{i=0}^{n-1} d_X(T^{n-i}y_i, T^{n-i-1}y_{i+1})<\sum_{i=0}^{n-1}\frac{\varepsilon}{2C}\le\frac{\varepsilon}{2}.
    \end{equation}
    Let $N\in \N$ and let $(x_n)_{n=1}^N$ be a $\delta$-partial pseudo-orbit in $X$.
    If $2\le N\le C$, we have
    \[\frac{1}{N-1}{\Big\lvert} \{1\le n\le N-1:d_X(Tx_n,x_{n+1})\ge\delta\}{\Big\rvert} < \delta < \frac{1}{C}\le\frac{1}{N},\]
    and hence the set on the left-hand side must be empty.
    Therefore, by Equation \eqref{eq delta psuedo orbit estimate}, we deduce that $(x_n)$ is $\varepsilon$-partially traced by $x_1$, and hence may assume $N>C$ for the remainder of the proof.
    Let $n_1<\dots<n_K$ be all the integers such that $d_X(Tx_{n_k-1}, x_{n_k})\ge \delta$; then $K< N\delta$. We also set $n_0= 1$ and $n_{K+1} = N+1$.
    We divide each integer interval $[n_k, n_{k+1})$ into subintervals of length $L$, separated by gaps of size $M$ between them (possibly leaving fewer than $L+M$ remaining integers at the end of $[n_k, n_{k+1})$), and construct a specification from these subintervals. 
    More precisely, for every $k=0,\dots, K$ such that $n_{k+1} - n _k \ge L + M$, we define $a_{k,i} = n_k + i (L+M)$, where $i$ takes all values $0,1,2\dots$ such that $a_{k,i} + L + M \le n_{k+1}$.
    Then $\{(T^{-a_{k,i}}x_{a_{k,i}}; a_{k,i}, a_{k,i}+L)\}_{k,i}$ (indexed such that $i$ varies first within each fixed $k$) is an $M$-spaced specification, so there exists $y\in X$ that $\varepsilon/4$-partially traces it\footnote{A priori, it is possible that the specification we defined is empty (if $n_{k+1} - n _k < L + M$ for all $k$), in which case we take an arbitrary $y\in X$.
    However, the estimate below, together with the assumption $N > C$, will indirectly rule out this case.}.
    For any tracing index $n\in [a_{k,i}, a_{k,i}+L)$ (Definition \ref{def spec and partial tracing}), we can use
    the $\varepsilon/4$-partial tracing together with Equation \eqref{eq delta psuedo orbit estimate} to obtain
    \[d_X(T^n y, x_n) \le  d_X(T^n y, T^{n-a_{k,i}}x_{a_{k,i}}) + d_X(T^{n-a_{k,i}}x_{a_{k,i}}, x_n) < \frac{\varepsilon}{4} + \frac{\varepsilon}{2} < \varepsilon.\]
    To estimate the number of integers $1\le n\le N$ for which the last equation holds, we briefly recall the construction: we divided the integers $1,\dots,N$ into at most $N\delta + 1$ intervals. Leaving aside less than $L+M$ points at the end of each interval, it was then subdivided into subintervals of length $L$, separated by gaps of size $M$. In each subinterval, the orbit of $y$ approximates all but at most $L\varepsilon/4$ of the $L$ points. Consequently,
    \begin{align*}
        \frac{1}{N}{\big\lvert} \{1\le n\le N:d_X(x_n,T^n y)\ge\varepsilon\}{\big\rvert}\le \frac{(N\delta+1)(L+M)}{N} + \frac{L\varepsilon/4+M}{L+M} \\
        \le \delta(L+M) + \frac{L+M}{N} + \frac{L\varepsilon/4+M}{L}<\frac{L+M}{C}+\frac{L+M}{C}+\frac{\varepsilon}{4} + \frac{M}{L}<\varepsilon,
    \end{align*}
    where the last two inequalities follow from our choice of $L,C$ and $\delta$ at the start of the proof. We conclude that $(x_n)$ is $\varepsilon$-partially traced by $y$.
\end{proof}
The following lemma describes a case in which a periodic point in a quotient of an abelian group dynamical system lifts to a periodic point in the group.
\begin{lemma}\label{lemma periodic point in coset}
    Let $(X,T)$ be an abelian group dynamical system, and $Y\le X$ a $T$-invariant closed subgroup on which $T$ acts ergodically. Suppose that for some $x\in X$ and $n\in \N$, the coset $x+Y$ is $T^n$-invariant. Then there exists $x'\in x+Y$ with $T^n x' = x'$.
\end{lemma}
\begin{proof}
    The core idea of this proof appears, in essence, in \cite[Lemma 10.9]{schmidt_book}.
    As mentioned in the proof of Lemma \ref{lemma ergodic subgroup and quotient}, the ergodicity of $(Y,T)$ implies that for every nontrivial character $\chi\in \widehat{Y}$ we have $\widehat{T}^n \chi \ne \chi$, and hence the homomorphism from $\widehat{Y}$ to $\widehat{Y}$ defined by $\chi\mapsto \widehat{T}^n \chi - \chi$ is injective. Therefore, the dual homomorphism of $Y$, which is defined by $y\mapsto T^n y - y$, is surjective.
    The $T^n$-invariance of $x+Y$ implies that $x-T^n x \in Y$, and hence we can find $y\in Y$ such that $T^n y-y= x-T^n x$, and we set $x' = x+y$.
\end{proof}
We are ready to prove the main result of this section, Proposition \ref{prop extension has partial spec}.
\begin{proof}[Proof of Proposition \ref{prop extension has partial spec}]
    Fix an $X$-invariant metric $d_X$ on $X$.
    Recall that we define $d_Y$ to be the restriction of $d_X$ to $Y$ and $d_{X/Y}$ as in Equation \eqref{eq metric on quotient}.
    Let $\varepsilon>0$. By Lemma \ref{lemma spec implies shadow} we can choose $0<\delta<\varepsilon/2$ such that every $\delta$-partial pseudo-orbit in $Y$ is $\varepsilon/2$-partially traced by some point in $Y$. Choose $0<\eta<\delta/4$ such that for every $x,y\in X$,
    \begin{equation}\label{eq eta delta}
        d_X(x,y)<\eta \Longrightarrow d_X(T x, T y)<\frac{\delta}{2}.
    \end{equation}
    Let $N_Y, M_Y\in \N$ and $N_{X/Y}, M_{X/Y}\in\N$ be the parameters provided by the partial specification for $Y$ and $X/Y$, respectively, corresponding to $\eta$ (Definition \ref{def partial specification}). Let $\PP\subseteq \N$ be the common collection of periods for $Y$ and $X/Y$ and set $N=\max\{N_Y, N_{X/Y}\}$ and $M=\max\{M_Y, M_{X/Y}\}$.
    
    Let $\{(x_i; a_i,b_i)\}_{i=1}^r$ be an $M$-spaced specification in $X$ and $p\in \mathcal{P}$ such that $p \ge \max\{N,(1+\varepsilon)b_r\}$; we aim to show that this specification can be $\varepsilon$-partially traced by a point of period $p$, and we briefly outline the three steps of the argument:
    \begin{enumerate}
        \item We project this specification to $X/Y$ and $\eta$-partially trace it with a point $x+Y$.
        \item To lift the approximation from $X/Y$ to $X$, we introduce elements $y_n\in Y$ to account for the discrepancy between $T^n x_i$ and $T^n x$. We will see that each sequence $y_{a_i},\dots,y_{b_i-1}$ forms a $\delta$-partial pseudo-orbit, which is $\varepsilon/2$-partially traced by some $z_i\in Y$.
        \item We use the partial specification in $Y$ to replace the sequence $(z_i)$ with a single point $z\in Y$, so that $x+z$ $\varepsilon$-partially traces the original specification.
    \end{enumerate}
    We now carry out the details of this argument. By the choice of $N$ and $M$, there exists $x\in X$ such that $x+Y$ $\eta$-partially traces the specification $\{(x_i + Y; a_i,b_i)\}_{i=1}^r$ in $X/Y$, and $T^p(x+Y)=x+Y$.
    
    Applying Lemma \ref{lemma periodic point in coset} and Remark \ref{remarks partial spec}(\ref{remark partial spec ergodic}) to replace $x$ with another point from the same coset if necessary, we may assume that $T^p x=x$.
    The definition of $d_{X/Y}$ implies that for every $i=1,\dots,r$ and $n=a_i,\dots,b_i$ there exists $y_n\in Y$ such that
    \[d_{X/Y} (T^n (x_i+ Y), T^n (x+Y)) = d_X(T^n(x_i-x),y_n).\]
    For each $i$, let $\Lambda_i$ be the corresponding set of tracing indices (Definition \ref{def spec and partial tracing}), i.e., those $n \in [a_i,b_i)$ for which the expression above is less than $\eta$. In the case where $b_i = a_i +1$, the single-point sequence $y_{a_i}$ is trivially a $\delta$-partial pseudo-orbit. Otherwise,
    if $n,n+1\in \Lambda_i$, then using Equation \eqref{eq eta delta} we obtain
    \[d_X(Ty_n, y_{n+1})\le d_X(Ty_n, T^{n+1}(x_i - x)) + d_X(T^{n+1}(x_i - x), y_{n+1})< \delta/2 + \eta <\delta.\]
    Moreover, since $\lvert \Lambda_i\rvert > (b_i-a_i)(1-\eta),$
    \begin{align*}
        {\big \lvert}\{a_i\le n < b_i - 1: d_Y(Ty_n, y_{n+1}) \ge \delta\}{\big \rvert} \le
        {\big \lvert}\{a_i\le n < b_i: n\notin \Lambda_i\text{ \;or \;} n+1\notin\Lambda_i\}{\big \rvert} \\
        \le 2(b_i-a_i-\lvert \Lambda_i\rvert) < 2\eta (b_i - a_i) < \delta (b_i - a_i - 1),
    \end{align*}
    which implies that the sequence $y_{a_i},\dots,y_{b_i-1}$ is a $\delta$-partial pseudo-orbit in $Y$. By the choice of $\delta$, this sequence can be $\varepsilon/2$-partially traced by some $z_i\in Y$, and we denote by $\Lambda_i'$ the corresponding collection of tracing indices (Definition \ref{def partial shadowing}). Now, the specification $\{(T^{-a_i}z_i;a_i,b_i)\}_{i=1}^r$ in $Y$ can be $\eta$-partially traced by some $z\in Y$ with $T^p z = z$, and for every $i$ let $\Lambda_i''$ denote the corresponding collection of tracing indices.
    Then $x+z$ is also of period $p$, and we claim that $x+z$\, $\varepsilon$-partially traces the original specification $\{(x_i; a_i,b_i)\}_{i=1}^r$. To see it, observe that for every $i$ and $n \in \Lambda_i\cap\Lambda_i'\cap\Lambda_i''$ we have
    \begin{align*}
        &d_X(T^n x_i, T^n(x+z)) = d_X (T^n(x_i - x), T^n z) \le \\
        &d_X(T^n(x_i - x), y_n) + d_X(y_n, T^{n-a_i}z_i) +
        d_X(T^{n-a_i}z_i, T^n z) < \eta + \varepsilon/2 + \eta < \varepsilon.
    \end{align*}
    It follows that
    \begin{align*}
        &{\big\lvert}\{a_i\le n< b_i: d_X(T^nx_i,T^n(x+z)) \ge \varepsilon\} {\big\rvert}\le \\
        &{\big\lvert}\{a_i\le n < b_i: n\notin \Lambda_i\cap\Lambda_i'\cap\Lambda_i'' \}{\big\rvert} \le 3(b_i - a_i) - (\lvert \Lambda_i\rvert  + \lvert \Lambda_i'\rvert + \lvert \Lambda_i''\rvert) < \varepsilon(b_i-a_i),
    \end{align*}
    which completes the proof.
\end{proof}
\section{Partial specification for abelian group dynamical systems}\label{section solenoid}
In this section, we show that ergodic abelian group dynamical systems satisfying the dcc also satisfy partial specification and admit dense ergodic periodic measures. The main remaining ingredient, addressed here, is to establish that the solenoids appearing in Proposition \ref{prop ergodic structure} satisfy partial specification. The proof relies on generalizations of ideas from smooth dynamics---particularly dynamics on the torus $\T^d$---to the settings of $p$-adic fields and their products. We begin by introducing the relevant notions.
\subsection{Linear maps over local fields}\label{subsection dynamics local fields}
In what follows, $d\in \N$ and $A\in \GL_d(\Q)$ is a matrix whose characteristic polynomial, denoted $f_A$, is irreducible over $\Q$. 
Let $p$ be a prime number or $\infty$ and let $K_p$ be a splitting field of $f_A$ over $\Q_p$ (recall that $\Q_\infty = \R$). There is a unique absolute value on $K_p$ which extends the $p$-adic absolute value on $\Q_p$, and we denote both of them by $\lvert \; \rvert_p$.
Since $f_A$ is irreducible over $\Q$, it has no multiple roots in $K_p$ and hence $A$ is diagonalizable over $K_p$.
For every eigenvalue $\lambda\in K_p$ of $A$, denote by $V_\lambda$ the corresponding eigenspace inside $K_p^d$. Summing over all eigenvalues $\lambda\in K_p$, we define
\begin{equation}\label{eq stable unstable in Kp}
    \widetilde{E^s_p} =\bigoplus_{\lvert \lambda \rvert_{p}<1} V_\lambda,\qquad
    \widetilde{E^c_p} = \bigoplus_{\lvert \lambda \rvert_{p}=1} V_\lambda, \qquad
    \widetilde{E^u_p} = \bigoplus_{\lvert \lambda \rvert_{p}>1} V_\lambda
\end{equation}
and define the \emph{$p$-adic stable, central and unstable subspaces} (\emph{with respect to $A$}) to be
\begin{equation}\label{eq def stable unstable subspaces}
    E^s_p = \Q_p^d\cap \widetilde{E^s_p}, \qquad
    E^c_p = \Q_p^d\cap \widetilde{E^c_p}, \qquad
    E^u_p = \Q_p^d\cap \widetilde{E^u_p}.
\end{equation}
These subspaces are $A$-invariant and independent of the choice of $K_p$.
Using the fact that $\lvert\sigma(\lambda)\rvert_{p} = \lvert \lambda\rvert_{p}$ for every $\sigma\in \gal(K_p/\Q_p)$, it can be shown that
\begin{equation}\label{eq Q_p^d decomposition}
        \Q_p^d = E_p^s \oplus E_p^c \oplus E_p^u,
\end{equation}
see \cite[Section II.1]{margulis}.

A norm can be defined for $p$-adic vector spaces, much like in the real and complex cases. We give the definition below and refer to \cite{nonarchimedean_func} for more details.
\begin{definition}
    Let $p$ be a prime number and $V$ a vector space over $\Q_p$. A \emph{norm} on $V$ is a function $\lVert\; \rVert: V\to \R_{\ge 0}$ such that for every $v,u\in V$ and $a\in \Q_p$,
    \begin{enumerate}
        \item $\lVert v \rVert=0$ if and only if $v=0$,
        \item $\lVert a v\rVert = \lvert a \rvert_p \lVert v\rVert$,
        \item $\lVert v + u\rVert \le \max\{\lVert v \rVert,\lVert u \rVert\}$.\footnote{In fact, the definition in \cite{nonarchimedean_func} requires slightly weaker assumptions, from which the properties described here still follow.}
    \end{enumerate}
\end{definition}
As in the Euclidean case, for any $d\in \N$ all norms on $\Q_p^d$ are equivalent and induce the standard topology on $\Q_p^d$ (\cite[Proposition 4.13]{nonarchimedean_func}).
It will be useful to adopt a norm on $\Q_p^d$ that is compatible with the decomposition of $\Q_p^d$ described in Equation \eqref{eq Q_p^d decomposition}.
\begin{proposition}[cf.\ {\cite{lind_dynamical}}, {\cite[Proposition 4.3]{padic_lectures}}, {\cite[Lemma 1.1]{margulis}}]\label{prop adapted norm}
    Let $d\in \N$, let $A\in \GL_d(\Q)$ be a matrix whose characteristic polynomial is irreducible over $\Q$ and let $p$ be a prime number or $\infty$. Then there exists a real number $\rho > 1$ and a norm $\lVert\;\rVert$ on $\Q_p^d$, which is said to be \emph{adapted to $A$}, with the following properties:
    \begin{enumerate}
        \item for every $x= x_s + x_c + x_u\in E_p^s \oplus E_p^c\oplus E_p^u = \Q_p^d$, $\lVert x \rVert = \max\{\lVert x_s \rVert,\lVert x_c \rVert,\lVert x_u \rVert\}$,
        \item for every $x\in E_p^s$, $\lVert Ax \rVert \le \rho^{-1}\lVert x \rVert$,
        \item for every $x\in E_p^u$, $\lVert Ax \rVert \ge \rho\lVert x \rVert$,
        \item for every $x\in E_p^c$, $\lVert Ax \rVert = \lVert x \rVert$.
    \end{enumerate}
\end{proposition}
\begin{proof}
    As above, let $K_p$ be a splitting field for the characteristic polynomial of $A$ over $\Q_p$, and let $v_1,\dots, v_d$ be an eigenbasis of $K_p^d$ for $A$. For every $v=\sum \alpha_i v_i \in \Q_p^d$ define $\lVert v\rVert = \max_i\{\lvert\alpha_i\rvert_{p}\}$. It is straightforward to verify that this defines a norm on $\Q_p^d$ with the desired properties.
\end{proof}
\subsection{Partial specification for automorphisms of solenoids}\label{subsection spec for solenoids}
Throughout this subsection, let $d\in \N$ and let $S$ be a set that consists of $\infty$ together with a finite (possibly empty) collection of prime numbers.
As before, $A$ is a rational matrix whose characteristic polynomial is irreducible over $\Q$, but we now require that $A\in \GL_d(\Z[\frac{1}{S}])$ (where $\Z[\frac{1}{S}]$ is the ring of rational numbers whose denominators are divisible only by primes in $S\setminus\{\infty\}$). We also require that none of the eigenvalues of $A$ is a root of unity.
Note that our assumptions on $A$ are consistent with the automorphism of the solenoid described in Proposition \ref{prop ergodic structure}.

Let us recall the notation from Subsection \ref{subsection solenloids notation}. We write $\Q_S^d = \prod_{p\in S} \Q_p^d$, and denote by $\Delta_S^d$ the diagonal embedding of $\Z[\frac{1}{S}]^d$ into $\Q_S^d$. The $d$-dimensional $S$-adic solenoid is then defined as $X_S^d = \Q_S^d / \Delta_S^d$.
The matrix $A$ acts diagonally on $\Q_S^d$ and on $X_S^d$, and recall that we identify the dual group $\widehat{\Q_S^d}$ with $\Q_S^d$ and the dual automorphism $\widehat{A}$ with (the diagonal action of) $A^\top$.

For every $p\in S$ let $E_p^s,E_p^c,E_p^u\subseteq \Q_p^d$ be the $p$-adic stable, central and unstable subspaces with respect to $A$ as in Equation \eqref{eq def stable unstable subspaces}, and define the following subsets of $\Q_S^d$:
\[E^s \coloneqq \prod_{p\in S}E_p^s, \qquad
E^c \coloneqq \prod_{p\in S}E_p^c, \qquad
E^u \coloneqq \prod_{p\in S}E_p^u.\]
\begin{lemma}[cf.\ {\cite[Chapter IV, Theorem 8]{weil}}]\label{lemma E^u ne 0}
    Under the above assumptions, $E^u \ne 0$.
\end{lemma}
\begin{proof}
    As before, for every $p\in S$ we let $K_p$ be a splitting field for the characteristic polynomial of $A$ over $\Q_p$, and let $\lambda_{p,1},\dots,\lambda_{p,r_p}$ be the eigenvalues of $A$ inside $K_p$ with $\lvert\lambda_{p,i}\rvert_p > 1$. We claim that such eigenvalues exist (i.e., $r_p>0$ for some $p$): let $\ell\in \N$ be the least common multiple of the denominators of the coefficients of the characteristic polynomial of $A$.
    If $\ell = 1$ then the polynomial has integer coefficients, and since we assume no eigenvalue of $A$ is a root of unity, a well known theorem of Kronecker implies that $\lvert \lambda\rvert_\infty > 1$ for some eigenvalue\footnote{We mention that, in contrast to the case of an integer matrix, a rational matrix may satisfy $E_\infty^s=E_\infty^u=0$; see for example \cite{lind_recurrent}.} $\lambda \in \C$. If $\ell > 1$, then the following stronger assertion is proven in \cite[pages 412, 416]{p_adic_entropy}:
    \[\prod_{p\in S\setminus\{\infty\}}\prod_{i=1}^{r_p} \lvert \lambda_{p,i}\rvert_p = \ell,\]
    so in particular $r_p > 0$ for some $p$. For this $p$, we have $\widetilde{E_p^u}\ne 0$ (Equation \eqref{eq stable unstable in Kp}) and hence also $E_p^u\ne 0$ and $E^u\ne 0$.
\end{proof}
For the rest of this section, let $\lVert\;\rVert_p$ be an adapted norm on $\Q_p^d$ as in Proposition \ref{prop adapted norm}, and define a metric on $\Q_S^d$ by
\begin{equation}\label{eq metric on Q_S^d}
    d_{\Q_S^d}(x,y)=\max_{p\in S} \{\lVert x_p - y_p\rVert_p\},\qquad x=(x_p)_{p\in S},y=(y_p)_{p\in S}\in \Q_S^d.
\end{equation}
For $\varepsilon>0$ and $n\in\N_0$, we denote
\begin{equation}\label{eq ball def}
    B(\varepsilon)= \{x\in\Q_S^d:d_{\Q_S^d}(x,0) <\varepsilon\},\qquad B(\varepsilon,n) = \bigcap_{j=0}^n A^{-j} {\big (} B(\varepsilon){\big )}, \qquad B^s(\varepsilon) = B(\varepsilon)\cap E^s,
\end{equation}
and $B^c(\varepsilon)$ and $B^u(\varepsilon)$ are defined analogously to $B^s(\varepsilon)$.
We list a few useful properties that follow directly from Proposition \ref{prop adapted norm} and the definition of $d_{\Q_S^d}$:
Choose a real number $\rho>1$ small enough to satisfy the conclusion of Proposition \ref{prop adapted norm} for all $p\in S$. Then for every $n\in \N_0$ and $\varepsilon>0$,
\begin{align}
    &A^n {\big (}B^s(\rho^n\varepsilon){\big )}\subseteq B^s(\varepsilon), \qquad A^{\pm n}{\big (}B^c(\varepsilon){\big )} =  B^c(\varepsilon), \qquad A^{-n} {\big (}B^u(\rho^n\varepsilon){\big )}\subseteq B^u(\varepsilon),\label{eq ball image under A} \\
    &B(\varepsilon,n) = B^s(\varepsilon) + B^c(\varepsilon) + A^{-n}(B^u(\varepsilon)). \label{eq n ball decomposition}
\end{align}
The following lemma shows that the projection of $E^u$ to the solenoid $X_S^d$ is dense.
\begin{lemma}
    $E^u + \Delta_S^d$ is dense in $\Q_{S}^d$.
\end{lemma}
\begin{proof}
    For a subgroup $H$ of $\Q_S^d$, denote by $H^\bot$ its annihilator in $\widehat{\Q_S^d}$.
    By duality theory, it suffices to show that $(E^u + \Delta_S^d)^\bot=0$.
    Since $(\Delta_S^d)^\bot = \widehat{X_S^d}$, we have, as stated in Equation \eqref{eq solenoid dual},
    \begin{equation}\label{eq annihilator of diagonal}
        (\Delta_S^d)^\bot = \{(-x,x,\dots,x):x\in \Z[\frac{1}{S}]^d\}\subseteq\R^d\times\prod_{p\in S\setminus\{\infty\}}\Q_p^d \simeq \widehat{\Q_S^d}.
    \end{equation}
    Moreover, writing $H^{\bot p}$ for the annihilator of a subgroup $H\le\Q_p^d$ in $\widehat{\Q_p^d}$ , we have
    \begin{equation}\label{eq annihilator of E^u}
        (E^u)^\bot = \prod_{p\in S} (E_p^u)^{\bot p}.
    \end{equation}
    Apply Lemma \ref{lemma E^u ne 0} to choose some $p\in S$ with $E_p^u\ne0$.
    Then $(E_p^u)^{\bot p}\cap \Q^d$ is an $A^\top$-invariant subspace of $\Q^d$, and since the characteristic polynomial of $A$ is irreducible over $\Q$, it must be either $0$ or $\Q^d$. The latter case is impossible since $E_p^u\ne0$, and therefore $(E_p^u)^{\bot p}\cap \Q^d = 0$. Combining this with Equations \eqref{eq annihilator of diagonal} and \eqref{eq annihilator of E^u}, we conclude that
    \begin{equation*}
        (E^u + \Delta_S^d)^\bot = (E^u)^\bot\cap (\Delta_S^d)^\bot = 0,
    \end{equation*}
and the proof is complete.
\end{proof}
\begin{corollary}\label{cor Q_s^d decomposition}
    For every $\varepsilon>0$ there exists $N\in\N$ such that for every $n\ge N$,
    \begin{equation*}
        \Q_S^d= A^n {\big(}B^u(\varepsilon){\big)} + B^s(\varepsilon) + B^c(\varepsilon) + \Delta_S^d.
    \end{equation*}
\end{corollary}
\begin{proof}
    Let $\varepsilon>0$. Since $\Delta_S^d$ is cocompact, the last lemma implies that there exists $R>0$ such that $B^u(R) + \Delta_S^d$ is $\varepsilon$-dense in $\Q_S^d$. It follows from Equation \eqref{eq n ball decomposition} (with $n=0$) that
    \[\Q_S^d = B^u(R) + \Delta_S^d + B(\varepsilon) \subseteq B^u(R+\varepsilon) + B^s(\varepsilon) + B^c(\varepsilon) + \Delta_S^d.\]
    By Equation \eqref{eq ball image under A},  $B^u(R+\varepsilon) \subseteq A^n {\big(}B^u(\varepsilon){\big)}$ for every sufficiently large $n\in \N$, and the assertion follows.
\end{proof}
The definition below is inspired by the construction of the set $D(\varepsilon)$ in Marcus's paper \cite{marcus}. Given a rational polynomial $f$, we aim to construct a large set of integers $n\in \N$ such that for each unimodular root $\lambda$ of $f$, the powers $\lambda^n$ remain uniformly bounded away from $1$ in various fields.
\begin{definition}\label{def bounded below set}
    Let $f\in \Q[x]$ and let $S$ be a set that consists of $\infty$ together with a finite collection of prime numbers.
    We say that a subset $\PP\subseteq \N$ is \emph{bounded below for $f$ and $S$} if it has bounded gaps, and if for every $p\in S$, there exists a splitting field $K_p$ of $f$ over $\Q_p$, such that
    \[\inf \lvert \lambda^n - 1 \rvert_p > 0,\]    
    where $\lvert\;\rvert_{p}$ is the unique absolute value on $K_p$ extending the $p$-adic absolute value on $\Q_p$, and the infimum is taken over all $n\in \PP$ and all roots $\lambda\in K_p$ of $f$ with $\lvert \lambda\rvert_{p} =1$.
\end{definition}
We will show in Corollary \ref{cor bounded below set} that such a set exists, provided all of the elements $\lambda$ above are not roots of unity.

The following is the main result of this subsection and will be proved at its end.
\begin{proposition}\label{prop solenoid spec}
    Let $d\in \N$, $S\subseteq \N\cup\{\infty\}$ and $A\in \GL_d(\Z[\frac{1}{S}])$ be as defined at the beginning of Subsection \ref{subsection spec for solenoids} and
    denote by $f$ the characteristic polynomial of $A$. Let $\PP\subseteq \N$ be bounded below for $f$ and $S$.
    Then the dynamical system $(X_S^d,A)$ satisfies partial specification with periods $\PP$.
\end{proposition}
In fact, we will prove a slightly stronger statement, but this version will suffice for our goal of establishing a similar result for more general groups.
The first step towards proving Proposition \ref{prop solenoid spec} is to show that every orbit in the solenoid can be approximated by a periodic one.
Let $d_{X_S^d}$ denote the quotient metric on $X_S^d$ induced by the metric $d_{\Q_S^d}$ from Equation \eqref{eq metric on Q_S^d}, as described in Equation \eqref{eq metric on quotient}.
\begin{proposition}[cf.\ {\cite[Proposition 3]{marcus}}]\label{prop closing lemma}
    Let $d\in\N$, $S\subseteq \N\cup\{\infty\}$ and $A\in \GL_d(\Z[\frac{1}{S}])$ be as defined at the beginning of Subsection \ref{subsection spec for solenoids} and denote by $f$ the characteristic polynomial of $A$. Let $\PP\subseteq \N $ be bounded below for $f$ and $S$ and let $0<\varepsilon<1$. Then every sufficiently large $n\in \PP$ satisfies the following: for every $x\in X_S^d$ there exists $y\in X_S^d$ such that $A^n y=y$ and $d_{X_S^d}(A^i x,A^i y)<\varepsilon$ for every integer $0\le i \le (1-\varepsilon)n$.
\end{proposition}
\begin{proof}
    Let $0<\delta<\varepsilon$ be sufficiently small and then let $N\in\N$ correspond to $\delta$ as in Corollary \ref{cor Q_s^d decomposition}. Next, let $n\in\PP$ be sufficiently large. The precise meaning of “sufficiently small” and “sufficiently large” will be clarified in the course of the proof. To begin with, we require that $n\delta\ge N$.
    Let $x\in X_S^d$ and choose $x'\in \Q_S^d$ which projects to $x$. By the choice of $N$, we can write $x'-A^n x' = v + w$, where $w\in \Delta_S^d$ and
    \begin{equation}\label{eq v decomposition}
        v = v_u + v_s + v_c \in A^{\lfloor n\delta\rfloor} {\big(}B^u(\delta){\big)} + B^s(\delta) + B^c(\delta).
    \end{equation}
    We claim that $(I-A^n)^{-1}v\in B(\varepsilon, n - \lfloor n\varepsilon\rfloor)$ (where $I$ is the identity matrix and $(I-A^n)$ is invertible since no eigenvalue of $A$ is a root of unity; see also Equation \eqref{eq ball def}). Let us first indicate how to complete the proof, assuming this holds. Set $y' = (I-A^n)^{-1}w\in \Q_S^d$. Then $A^n y' = y' - w$, which means that $y'$ projects to a point of period $n$ in $X_S^d$.
    Furthermore,
    \[x' - y' = (I-A^n)^{-1}v\in B(\varepsilon, n-\lfloor n\varepsilon\rfloor),\]
    and hence the projection of $y'$ to $X_S^d$ has the desired properties.
    
    We are left with proving that $(I-A^n)^{-1}v\in B(\varepsilon, n - \lfloor n\varepsilon\rfloor)$. By Equation \eqref{eq n ball decomposition} it suffices to show that
    \begin{equation}\label{eq inclusions to prove}
        (I-A^n)^{-1}v_s \in B^s(\varepsilon),\quad (I-A^n)^{-1}v_c \in B^c(\varepsilon),\quad (I-A^n)^{-1}v_u\in A^{-n + \lfloor n\varepsilon\rfloor}B^u(\varepsilon).
    \end{equation}
    Notice that $A^m y \underset{m\to\infty}{\to}0$ uniformly on $B^s(\delta)$, and hence if $\delta$ is small enough and $n$ is large enough then $(I-A^n)^{-1}v_s \in B^s(\varepsilon)$.
    By the same reasoning, $(A^{-n}-I)^{-1}B^u(\delta)\subseteq B^u(\varepsilon)$. Thus,
    \[(I-A^n)^{-1}A^{\lfloor n\delta\rfloor}{\big(}B^u(\delta){\big)} = A^{-n + \lfloor n\delta\rfloor}(A^{-n}-I)^{-1}{\big(}B^u(\delta){\big)}\subseteq A^{-n + \lfloor n\delta\rfloor}{\big(}B^u(\varepsilon){\big)},\]
    and in particular by Equation \eqref{eq v decomposition} (and also Equation \eqref{eq ball image under A} to replace $n\delta$ with $n\varepsilon$), $(I-A^n)^{-1}v_u\in A^{-n + \lfloor n\varepsilon\rfloor}{\big(}B^u(\varepsilon){\big)}$.
    
    We now proceed to prove the final assertion of Equation \eqref{eq inclusions to prove}, concerning $B^c(\varepsilon)$. For every $p\in S$, let $K_p$ be the splitting field of $f$ over $\Q_p$ that appears in the definition of $\PP$ (Definition \ref{def bounded below set}), and $\lambda_{p,1},\dots,\lambda_{p,r_p}\in K_p$ the unimodular roots of $f$.
    Assume first that $r_p > 0$. Since we assume $f$ is irreducible over $\Q$, it has no multiple roots, and hence we can choose $M_p\in \GL_{r_p}(K_p)$ that diagonalizes $\restr{A}{E^c}$. Then
    \[(I-\restr{A^n}{E^c})^{-1}=M_p^{-1} \begin{pmatrix}
        (1-\lambda_{p,1}^n)^{-1} \\ & \ddots \\ & & (1-\lambda_{p,r_p}^n)^{-1}
    \end{pmatrix}M_p\]
    (cf.\ \cite[Proposition 2.2]{rigidity_properties}, \cite{berend}).
    By the definition of $\PP$,
    \[\sup_{\substack{1\le i\le r_p \\ m\in\PP}}\lvert (1-\lambda_{p,i}^m)^{-1}\rvert_p<\infty,\]
    and thus the collection of matrices $\{(I-\restr{A^m}{E^c})^{-1}:m\in\PP\}$ is contained in a compact subset of matrices over $K_p$.
    Denote $B_p^c(\delta) = \{z\in E_p^c:\lVert z\rVert_p<\delta\}$. Then for sufficiently small $\delta$ (independent of $n$), we have
    $(I-A^n)^{-1}{\big(}B_p^c(\delta){\big)} \subseteq B_p^c (\varepsilon)$.
    If $r_p = 0$, then $E_p^c = 0$ and the last inclusion holds trivially.
    Since this inclusion holds for all $p\in S$,
    we conclude that $(I-A^n)^{-1}{\big(}B^c(\delta){\big)} \subseteq B^c (\varepsilon)$,
    verifying Equation \eqref{eq inclusions to prove} and completing the proof.
\end{proof}
The next step towards Proposition \ref{prop solenoid spec} addresses the tracing part of the partial specification, setting aside the periodicity requirement. Similar results appear in \cite[Theorem 6.5]{lind_split} and in \cite[Theorem 2]{solenoidal} (see also \cite{dateyama}), though these references work within a framework that differs substantially from ours, using a different model for the solenoid. Since the tools we have developed so far allow for a short proof, closer in spirit to \cite[Lemma 2.1]{marcus}, we include it here for completeness.
\begin{proposition}\label{prop nonperiodic spec}
    Let $d\in\N$, $S\subseteq \N\cup\{\infty\}$ and $A\in \GL_d(\Z[\frac{1}{S}])$ be as defined at the beginning of Subsection \ref{subsection spec for solenoids}. Then for every $\varepsilon>0$ there exists $M\in \N$ such that every $M$-spaced specification in $X_S^d$ can be $\varepsilon$-partially traced (with respect to $A$).
\end{proposition}
\begin{proof}
    Let $\varepsilon>0$ and set $\delta = (1-\rho^{-1})\varepsilon$, where $\rho>1$ is as in Equation \eqref{eq ball image under A}.
    Apply Corollary \ref{cor Q_s^d decomposition} to find $M\in \N$ such that for every $m\ge M$,
    \begin{equation}\label{eq decomposition for M}
        \Q_S^d= A^m (B^u(\delta)) + B^s(\delta) + B^c(\delta) + \Delta_S^d.
    \end{equation}
    Denote by $\pi:\Q_S^d\to X_S^d$ the canonical projection, and for every $x \in X_S^d$ and $n\in \N_0$ consider the subsets of $X_S^d$ defined by
    \[\mathcal{C}(x,n) = A^n{\big(}(x+ \pi(B^u(\delta)){\big)},\qquad \mathcal{D}(x,n) = x + \pi(B(\delta,n))\]
    (see Equation \eqref{eq ball def}).
    We claim that for every $x,y\in X_S^d$ and every $n,m\in \N$ with $m\ge M$ we have $\mathcal{C}(x,m)\cap \mathcal{D}(y, n)\ne\emptyset$. Indeed, choose preimages $x',y'\in \Q_S^d$ of $x,y$. By Equation \eqref{eq decomposition for M}, there exist $z_s\in B^s(\delta)$ and $z_c\in B^c(\delta)$ such that
    \[y'-A^mx' + z_s + z_c \in A^m (B^u(\delta)) + \Delta_S^d.\]
    Projecting the last equation to $X_S^d$ shows that $y + \pi(z_s + z_c) \in \mathcal{C}(x,m)$, and Equation \eqref{eq n ball decomposition} implies that also $y + \pi(z_s + z_c) \in \mathcal{D}(y,n)$.
    
    Let $\{(x_i; a_i,b_i)\}_{i=1}^r$ be an $M$-spaced specification in $X_S^d$.
    We now define a sequence of points $y_1,\dots,y_r\in X_S^d$ by induction on $k\le r$, such that for every $1\le i \le k$ and $a_i\le n < b_i$,
    \[d_{X_S^d}(A^n x_i, A^n y_k) < \delta S_{k-i}\quad\text{where}\quad S_\ell = \sum_{j=0}^{\ell}\rho^{-j}.\]
    This will conclude the proof, since $y_r$ will satisfy $d_{X_S^d}(A^n x_i, A^n y_r)< \delta (1-\rho^{-1})^{-1}=\varepsilon$.
    For $k=1$ we can simply take $y_1 = x_1$. Assume now $y_k$ has been defined. Since $a_{k+1}-b_k \ge M$, we have
    \begin{equation}\label{eq big intersection}
        \mathcal{C}(A^{b_k}y_k, a_{k+1}-b_k)\cap \mathcal{D}(A^{a_{k+1}}x_{k+1}, b_{k+1}-a_{k+1})\ne\emptyset,
    \end{equation}
    so we may choose a point $y_{k+1}$ such that $A^{a_{k+1}}y_{k+1}$ lies in the above intersection.
    The membership in the latter set implies that for every $a_{k+1}\le n < b_{k+1}$,
    \[d_{X_S^d}(A^n x_{k+1}, A^n y_{k+1}) < \delta.\]
    Let $1\le i\le k$ and $a_i\le n < b_i$. Unpacking the definition of the former set in Equation \eqref{eq big intersection} and applying $A^{-(a_{k+1}-n)}$ to both sides yields
    \[ A^n y_{k+1} \in A^n y_k + A^{-(b_k-n)}\pi(B^u(\delta))\subseteq A^n y_k + \pi(B^u(\rho^{-(b_k-n)}\delta)),\]
    where we used Equation \eqref{eq ball image under A} to obtain the last inclusion.
    Notice that $b_k - n \ge b_k-b_i + 1\ge k - i + 1$, and thus from the induction hypothesis and the last equation we obtain
    \begin{align*}
        d_{X_S^d}(A^nx_i,A^n y_{k+1})&\le d_{X_S^d}(A^n x_i, A^n y_k) + d_{X_S^d}(A^n y_k, A^n y_{k+1}) \\
        &< \delta S_{k-i} + \delta\rho^{-(k+1-i)}=\delta S_{k+1-i},
    \end{align*}
    completing the construction of the sequence $(y_k)$ and the proof.
\end{proof}
All that remains to prove Proposition \ref{prop solenoid spec} is to combine the last two propositions.
\begin{proof}[Proof of Proposition \ref{prop solenoid spec}]
    Let $\varepsilon>0$, and we may assume that $\varepsilon<1$. Let $M\in \N$ be provided by Proposition \ref{prop nonperiodic spec} for $\varepsilon/2$, and $N\in\N$ large enough so that Proposition \ref{prop closing lemma} holds for $\varepsilon/2$ and for every $n\in \PP$ with $n\ge N$. Suppose we are given an $M$-spaced specification $\xi=\{(x_i;a_i,b_i)\}_{i=1}^r$ in $X_S^d$ and some $n\in \PP$ such that $n\ge \max\{N,(1+\varepsilon)b_r\}$. By Proposition \ref{prop nonperiodic spec}, there exists $x\in X_S^d$ that $\varepsilon/2$-partially traces $\xi$. Next we use Proposition \ref{prop closing lemma} to find $y\in X_S^d$ of period $n$ such that $d_{X_S^d}(A^ix,A^iy)<\varepsilon/2$ for every $0\le i \le (1-\varepsilon/2)n$.
    But $(1-\varepsilon/2)n\ge (1-\varepsilon/2)(1+\varepsilon)b_r\ge b_r$, and it follows that $\xi$ is $\varepsilon$-partially traced by $y$.
\end{proof}
We now turn to proving the existence of bounded below sets (Definition \ref{def bounded below set}).
It will also be important later that such sets can be chosen to lie in $c\N$ for any $c\in\N$.
\begin{lemma}
    Let $G$ be an abelian group, written in multiplicative notation, with a $G$-invariant metric $d_G$. Suppose $g_1,\dots,g_r\in G$ are elements of infinite order. Then for every $c\in\N$ there exists a subset $\PP\subseteq c\N$ with bounded gaps such that
    \[\inf_{\substack{1\le i\le r \\ n\in\PP}} d_G(g_i^n,1) > 0.\]
\end{lemma}
\begin{proof}
    Let $c\in \N$.
    It suffices to construct an increasing sequence $(n_k)_{k=1}^\infty$ in $\N$ with bounded gaps, such that
    \begin{equation}\label{eq infimum}
        \inf_{\substack{1\le i\le r \\ k\in\N}}d_G(g_i^{cn_k}, 1) > 0,
    \end{equation}
    since we can then set $\PP = \{cn_k :k\in \N\}$.
    For $\delta >0$ denote by $B(\delta)\subseteq G$ the open ball of radius $\delta$ centered at $1$.
    As the elements $g_i$ have infinite order, we can choose $\delta > 0$ such that for every $i=1,\dots,r$ and $m=1,\dots,r+1$,
    \begin{equation}\label{eq ball translation}
        B(\delta)\cap g_i^{cm}B(\delta)=\emptyset.
    \end{equation}
    In particular, for every $n\in \N$ and $i=1,\dots,r$,
    \begin{equation}\label{eq delta rotation2}
        {\Big\lvert}\{1\le m \le r+1 : g_i^{cm+cn}\in B(\delta)\}{\Big\rvert} \le 1.
    \end{equation}
    We construct a sequence $(n_k)$ with bounded gaps inductively so that for every $i=1,\dots,r$, we have $g_i^{cn_k}\notin B(\delta)$. By Equation \eqref{eq ball translation} we may take $n_1=1$.
    Suppose $n_k$ has been defined. Then, by Equation \eqref{eq delta rotation2}, for each $i=1,\dots,r$, there exists at most one integer $1 \le m_i \le r + 1$ such that $g_i^{c(n_k+m_i)} \in B(\delta)$. Hence, we can choose an integer $1 \le m \le r+1$ with $m\ne m_1,\dots, m_r$, and set $n_{k+1} = n_k + m$, completing the construction. It follows that the left-hand side of Equation \eqref{eq infimum} is bounded below by $\delta$.
\end{proof}
\begin{corollary}\label{cor bounded below set}
    Suppose that $f\in \Q[x]$ has no roots of unity among its roots and let $S$ be a set that consists of $\infty$ together with a finite collection of prime numbers. Then for any $c\in \N$, there exists a subset $\PP\subseteq c\N$ which is bounded below for $f$ and $S$.
\end{corollary}
\begin{proof}
    As in Definition \ref{def bounded below set}, for every $p\in S$ let $K_p$ be a splitting field of $f$ over $\Q_p$, with the absolute value $\lvert\;\rvert_{p}$.
    Denote by $\lambda_{p,1},\dots,\lambda_{p,r_p}\in K_p$ the roots of $f$ which belong to the unit group $G_p\coloneqq \{x\in K_p: \lvert x\rvert_{p} = 1\}$.
    Set $G\coloneqq\prod_{p\in S} G_p$ and define a $G$-invariant metric on $G$ by $d((x_p)_p,(y_p)_p) = \max_p \{\lvert x_p - y_p\rvert_{p}\}$. Applying the previous lemma to the group $G$ and the images of the elements $\{\lambda_{p,i}:p\in S, i=1,\dots,r_p\}$ under the canonical embeddings $G_p \hookrightarrow G$ yields the desired set.
\end{proof}
\subsection{Partial specification for abelian group dynamical systems}
We begin with a simple lemma about partial specification in general dynamical systems.
\begin{lemma}\label{lemma spec of factor}
    Let $(X,T)$ be a dynamical system and $(Y,S)$ a continuous factor. If $(X,T)$ satisfies partial specification with periods $\PP\subseteq \N$, then so does $(Y,S)$.
\end{lemma}
\begin{proof}
    Fix metrics $d_X$ and $d_Y$ on $X$ and $Y$, and let $\phi: X\to Y$  be the continuous factor map. Let $\varepsilon>0$ and choose $0<\delta <\varepsilon$ such that for every $x_1,x_2\in X$,
    \[d_X(x_1,x_2)<\delta \Longrightarrow d_Y(\phi(x_1), \phi(x_2))<\varepsilon.\]
    Let $N,M\in \N$ be the parameters provided by the partial specification for $X$, corresponding to $\delta$ (Definition \ref{def partial specification}).
    Given an $M$-spaced specification $\xi=\{(y_i;a_i,b_i)\}_{i=1}^r$ in $Y$, we can choose preimages $x_i\in \phi^{-1}(y_i)$ and lift $\xi$ to an $M$-spaced specification $\xi'=\{(x_i;a_i,b_i)\}_{i=1}^r$ in $X$. Thus, if $\PP\ni n\ge \max\{N, (1+\varepsilon)b_r\}$ we can find $x\in X$ of period $n$ which $\delta$-partially traces $\xi'$, and $\phi(x)$ will $\varepsilon$-partially trace $\xi$.
\end{proof}
We are ready to prove Theorem \ref{thm spec intro}.
\begin{proof}[Proof of Theorem \ref{thm spec intro}]
    Pick an ergodic abelian group dynamical system $(X',T')$, of which $(X,T)$ is a continuous factor, and $T'$-invariant closed subgroups $0=X_n\le\cdots\le X_0 = X'$ as described in Proposition \ref{prop ergodic structure}.
    For every quotient $(X_k/X_{k+1}, T')$ not isomorphic to a Bernoulli shift, apply case \eqref{prop ergodic structure, solenoid} of Proposition \ref{prop ergodic structure} to select suitable $S_k\subseteq \N\cup\{\infty\}$ and $A_k \in \GL_{d_k}(\Q)$ such that $(X_k/X_{k+1}, T')$ is a continuous factor of the $S_k$-adic solenoid $(X_{S_k}^{d_k}, A_k)$. Let $f_k$ denote the characteristic polynomial of $A_k$, and set $S=\bigcup_k S_k$ and $f=\prod_k f_k$.
    
    Fix some $c\in \N$, and apply Corollary \ref{cor bounded below set} to find $\PP\subseteq c\N$ which is bounded below for $f$ and $S$. Clearly, $\PP$ is also bounded below for every pair $f_k,S_k$, and hence by Proposition \ref{prop solenoid spec} together with Lemma \ref{lemma spec of factor}, $(X_k/X_{k+1}, T')$ satisfies partial specification with periods $\PP$.
    By Example \ref{example bernoulli has spec}, any quotient $(X_j/X_{j+1}, T')$ that is isomorphic to a Bernoulli shift also satisfies partial specification with periods $\PP$. Proposition \ref{prop extension has partial spec} then implies that $(X',T')$ satisfies partial specification with periods $\PP$, and a further application of Lemma \ref{lemma spec of factor} shows that $(X,T)$ does as well. Since $c$ was arbitrary, $(X,T)$ satisfies partial specification.
\end{proof}
Recall that for a dynamical system $(X,T)$, we write $\M(X)$ for the space of Borel probability measures on $X$, endowed with the weak-* topology, and $\M_T(X)$ for the closed subspace of $T$-invariant measures. We denote by $\delta_x$ the Dirac measure at $x\in X$, and for $n\in \N$ we denote $\mathfrak{m}_{T,x,n} = \frac{1}{n}\sum_{i=0}^{n-1} \delta_{T^i x}$.

Recall in addition that for $\mu\in \M_T(X)$, $x\in X$ is called \emph{$\mu$-generic} if $\mathfrak{m}_{T,x,n}\underset{n\to\infty}{\longrightarrow}\mu$, and that if $\mu$ is ergodic, then $\mu$-a.e.\ $x\in X$ is $\mu$-generic as a consequence of Birkhoff's ergodic theorem.

We now restate the ergodic part of Theorem \ref{thm dpm intro}.
\begin{theorem}\label{thm dense ergodic measures}
    Let $(X,T)$ be an ergodic abelian group dynamical system satisfying the dcc. Then the ergodic finitely supported measures in $\M_T(X)$ are dense in $\M_T(X)$.
\end{theorem}
\begin{proof}
    Many periodic specification-like properties are known to imply this assertion, with the details of the proof adapted to the specific situation; see e.g.\ \cite{marcus, sigmund} and \cite[Theorem 1.1(1)]{on_density}. Our case is no different, and the argument proceeds in essentially the same way. For this reason we only sketch the proof and omit the technical details.
    
    Fix a metric $d$ on $\M_T(X)$. Let $\mu\in \M_T(X)$ and $\varepsilon>0$, then there exist ergodic measures $\mu_1,\dots,\mu_k\in \M_T(X)$ and a convex combination $\sum_{i=1}^k t_i\mu_i$ such that $d(\mu, \sum t_i\mu_i)<\varepsilon$. For $i=1,\dots, k$, let $x_i\in X$ be a $\mu_i$-generic point. Pick $N\in \N$ such that for every $i$ and $n\ge N$, $d(\mu_i,\m_{T,x_i,n})<\varepsilon$.
    Consider the orbit segments $x_i,\dots,T^{n_i}x_i$, where $n_i\ge N$ and the ratio ${n_i}/(n_1+\cdots+n_k)$ approximates its weight $t_i$.
    Using these segments, we construct a suitably spaced specification, and apply the partial specification of $(X,T)$ to obtain a periodic point $x\in X$ whose orbit approximates them, except on a negligible set. The $T$-invariant measure supported on the finite orbit of $x$ then approximates $\mu$.
\end{proof}
\section{Unipotent toral automorphisms}\label{section unipotent}
In this section we study the ergodic measures for a unipotent toral automorphism, i.e., a toral automorphism induced by a unipotent matrix.
We then use this analysis to show that such automorphisms admit a strong form of dense periodic measures.
This, in turn, allows us to deduce a property analogous to partial specification.

\subsection{Ergodic measures for unipotent toral automorphisms}
We denote the support of a measure $\mu$ by $\supp(\mu)$. Recall that if $(X,T)$ is a dynamical system and $\mu\in \M_T(X)$ then $\supp(\mu)$ is a $T$-invariant closed subset of full measure.
As is common in unipotent homogeneous dynamics---unlike in the previous sections---in this case, the invariant measures are fully classified.
\begin{proposition}\label{prop measure classification}
    Let $d \in \N$ and let $U \in \GL_d(\Z)$ be a unipotent matrix. Let $\mu \in \M_{U}(\T^d)$ be an ergodic measure.
    Then there exist $m\in\N$, $a\in \T^d$ and a closed connected proper subgroup $H\le\T^d$ such that
    $\supp(\mu)=\bigcup_{i=0}^{m-1} U^i a + H$,
    and each $\restr{\mu}{U^i a +H}$ is a translate of the (appropriately normalized) Haar measure on $H$.
\end{proposition}
The key idea in the following proof is due to Levit and Vigdorovich.
\begin{proof}
    Consider the group $G = \SL_d(\R)\ltimes \R^d$ (with the standard action of $\SL_d(\R)$ on $\R^d$) and its lattice $\Gamma = \SL_d(\Z)\ltimes \Z^d$. We adopt the convention that for $(A,x),(B,y)\in G$, multiplication is given by $(A,x)(B,y)=(AB,x+Ay)$, and denote the identity matrix by $1$.
    The map $\phi: \T^d \to G/\Gamma$ defined for $x\in \R^d$ by $\phi(x+\Z^d) = (1,x)\Gamma$ is a homeomorphism on its image
    \[\phi(\T^d) = \{(1, x)\Gamma :x \in \R^d\} \subseteq G/\Gamma.\]
    Moreover, for $A\in \SL_d(\Z)$ and $x\in \R^d$ we have
    \[(A,0)\phi(x+\Z^d) = (A, Ax)\Gamma = (1, Ax)(A,0)\Gamma = \phi(Ax+\Z^d),\]
    so $\phi$ is $\SL_d(\Z)$-equivariant. In particular, given an ergodic measure $\mu\in \M_{U}(\T^d)$ we obtain an ergodic $U$-invariant measure $\nu\coloneqq\phi_*\mu$ on $\phi(\T^d)$. By Shah's theorem \cite[Theorem 1.1]{shah}, which generalizes Ratner's classical measure classification theorem, there exists a subgroup $S< G$ and $x\in \R^d$ such that $\nu$ is the $S$-invariant measure supported on the orbit $S(1,x)$.
    Identifying $\R^d$ with its image in $G$, we define
    \[ \widetilde{H} = \{g\in \R^d : g_*\nu = \nu\}.\]
    $\R^d$ is a normal subgroup of $G$, and it follows that $S$ normalizes $\widetilde{H}$.
    Replace $S$ with the subgroup of $G$ generated by $S$ and $\widetilde{H}$; then $\widetilde{H} = S\cap \R^d$ is a normal subgroup of $S$.
    Next, notice that for every $y\in \R^d$ and \textit{non}-integer matrix $A \in \SL_d(\R)$, $(A,y)\phi(\T^d)$ is disjoint from $\phi(\T^d)$, and hence $S \subseteq \SL_d(\Z)\ltimes \R^d$. Consequently, $S/\widetilde{H}$ can be identified with a subgroup of $\SL_d(\Z)$.
    
    We are now in position to apply Lemma \ref{lemma finitely many orbits}, stated below,
    which asserts that $\nu$ is supported on finitely many $\widetilde{H}$-orbits. Thus, applying $\phi^{-1}$ shows that the original measure $\mu$ on $\T^d$ is supported on finitely many cosets of $H'$, 
    where $H'$ is the closure of the projection of $\widetilde{H}$ to $\T^d$. In addition, the $\widetilde{H}$-invariance of $\nu$ implies that the restriction of $\mu$ to each of those cosets is $H'$-invariant, and therefore must be a translate of the Haar measure on $H'$.
    Finally, we define $H$ to be the identity connected component of $H'$. $\mu$ remains a translate of Haar measure on each of the resulting finitely many cosets of $H$.
    That the cosets can be written as $a+H,\dots,U^{m-1} a+H$ follows from the fact that $U$ acts transitively on them, as $\mu$ is ergodic.
    Finally, notice that $H$ must be a proper subgroup of $\T^d$ since the Haar measure on $\T^d$ is not $U$-ergodic.
\end{proof}
In the proof of the preceding proposition, we made use of the following lemma.
\begin{lemma}\label{lemma finitely many orbits}
    Let $S$ be a group acting measurably and transitively on a Borel space $X$.
    Let $\nu$ be an $S$-invariant probability measure on $X$, and suppose that $H$ is a normal subgroup of $S$ such that $S/H$ is countable and such that the $H$-orbits in $X$ are measurable. Then $\nu$ is supported on finitely many $H$-orbits.
\end{lemma}
\begin{proof}
    $S$ acts on the orbit space $H\setminus X$: for $s\in S$ and an orbit $Hx$ we have $sH.x=Hs.x$. The transitivity of the $S$-action on $X$ implies that this action is transitive as well. Therefore, by the $S$-invariance of $\nu$, all $H$-orbits must have the same measure.
    Since $S/H$ is countable, we can choose $s_1,s_2,\ldots \in S$ and $x\in X$ such that $X = S.x = \bigsqcup_n Hs_n.x$, where the union is either finite or countably infinite. Then
    \[1 = \sum_n \nu(Hs_n.x) = \sum_n \nu(Hs_1.x).\]
    This equation is valid only when the sum contains finitely many positive terms, from which the claim follows.
\end{proof}
While a unipotent automorphism of $\T^d$ is not uniquely ergodic, its restriction to certain subsets can be.
\begin{proposition}\label{prop unique ergodicity}
    Let $U\in \GL_d(\Z)$ be a unipotent matrix and $\mu\in \M_{U}(\T^d)$ an ergodic measure, whose support consists of $m\in \N$ connected components.
    Then for every $n\in\N$ coprime to $m$, the dynamical system
    $(\supp(\mu), \restr{U^n}{\supp(\mu)})$ is uniquely ergodic, with $\mu$ being the unique ergodic measure.
\end{proposition}
\begin{proof}
    By Proposition \ref{prop measure classification}, $\supp(\mu)=  \bigcup_{i=0}^{m-1}U^i a + H$ for some $a\in \T^d$ and a closed connected subgroup $H\le \T^d$.
    We begin by fixing $n\in\N$ coprime to $m$ and showing that $\mu$ is $U^n$-ergodic.
    The ergodic decomposition $\mu=\int \nu\; d\lambda(\nu)$ with respect to $U^n$ yields a representation of $\mu$ as an average of $U$-invariant measures $\mu = \int \frac{1}{n}\sum_{i=0}^{n-1}U_*^i\nu\; d\lambda(\nu)$.
    But since $\mu$ is $U$-ergodic, there exists a $U^n$-ergodic measure $\nu\in \M_{U^n}(\T^d)$ such that $\mu=\frac{1}{n}\sum_{i=0}^{n-1}U_*^i\nu$. Applying Proposition \ref{prop measure classification} to $U^n$ and $\nu$, we can write $\supp(\nu)=  \bigcup_{i=0}^{m'-1}U^{ni} a' + H'$, with $m'\in\N$, $a'\in \T^d$ and $H'\le \T^d$ being a closed connected subgroup. The equality of the supports of $\mu$ and $\frac{1}{n}\sum_{i=0}^{n-1}U_*^i\nu$ implies that $a' + H'\subseteq \bigcup_i U^i a + H$.
    The connectedness of $H'$ then implies that $H'$ is contained in a single coset of $H$, which  must be $H$; thus, $H'\subseteq H$. An analogous argument gives $H\subseteq H'$, and hence $H=H'$.
    
    It follows that $\supp(\nu)$ consists of a subset of the $H$-cosets comprising $\supp(\mu)$.
    But since $\mu$ is $U$-ergodic and $n$ is coprime to $m$, $U^n$ acts transitively on the collection of cosets $\{U^i a + H\}_{i=0}^{m-1}$, and hence $\supp(\nu)=\supp(\mu)$. As both measures are given by the translate of the Haar measure to each coset, we conclude that $\mu=\nu$, and $\mu$ is $U^n$-ergodic.    

    We turn to prove that $(\supp(\mu), \restr{U}{\supp(\mu)})$ is uniquely ergodic. First let us assume that $\supp(\mu)= a + H$ is connected.
    Notice that
    \[a + H = U(a + H) = Ua + UH,\]
    and since the subgroups $H, UH$ differ by a translation, they must be equal, namely $UH = H$ and $Ua - a \in H$.
    Consider the affine transformation $T:H\to H$ defined by $T(h) = Uh + Ua - a$. Observe that the translation map $\ell_a(x)=x+a$ satisfies $\ell_a \circ T = U\circ \ell_a$ on $H$, and it takes the Haar measure on $H$ to $\mu$, so the unique ergodicity of $\restr{U}{a+H}$ will follow once we show $T$ is uniquely ergodic.
    As $H$ is closed and connected, the dynamical system $(H,T)$ is isomorphic to a torus $\T^k$ for some $k\le d$ equipped with an affine transformation whose linear part is unipotent. Moreover, this system is ergodic with respect to Haar measure on $\T^k$ as it is conjugate to the ergodic system $(\T^d,U,\mu)$.
    The study of unique ergodicity for affine transformations of the torus was carried out by Hahn; in particular, \cite[Theorems 6 and 18]{hahn} show that an ergodic affine transformation with a unipotent linear part is uniquely ergodic.
    
    Now if $\supp(\mu)$ is not connected, we still have $\supp(\mu)=\bigcup_{i=0}^{m-1}E_i$, where $E_i\coloneqq U^i a + H$.
    Since $\mu$ is ergodic, the induced transformation $U^m$ on each coset $E_i$ is ergodic with respect to $\frac{1}{\mu(E_i)}\restr{\mu}{E_i}$ (see \cite[Lemma 2.43]{ergodic_view}). Thus, by what we have just proved, this is the unique $U^m$-invariant measure on $E_i$, and it follows that $\mu$ is the unique $U$-invariant measure on $\supp(\mu)$.

    Finally, if $n$ is coprime to $m$, then $\mu$ is $U^n$-ergodic by the first part of the proof. Now the second part, applied to $U^n$, shows that $\restr{U^n}{\supp(\mu)}$ is uniquely ergodic.
\end{proof}
\subsection{Dense periodic measures for unipotent toral automorphisms}\label{subsec dpm unipotent}
Our next goal is to establish a strong form of dense periodic measures for unipotent toral automorphisms.
It asserts that for any ergodic measure $\mu$, there \emph{exists} $c\in \N$ such that $\mu$ can be approximated by a measure supported on the orbit of a point whose period is \emph{any} sufficiently large element of $c\N$.
This contrasts with the situation in the case of partial specification, discussed in previous sections, where for \emph{any} given $c\in \N$, the set of periods could be chosen to be \emph{some} subset of $c\N$, but the specific elements within that set could not be predetermined.
This distinction will allow us to construct periodic points in group extensions where one factor is a unipotent automorphism and the other satisfies partial specification.
\begin{lemma}\label{lemma periodic in subgroup}
    Let $U\in \GL_d(\Z)$ be a unipotent matrix. Let $X\le \T^d$ be an infinite closed subgroup and let $x\in X$.
    Then there exist $c\in \N$ and a sequence $(x_n)_{n=1}^\infty$ in $X$ such that $x_n$ is of period $cn$ and $x_n\underset{n\to\infty}{\longrightarrow}x$.
\end{lemma}
\begin{proof}
    First, let us assume that $U$ is in Jordan normal form.
    Since $X/X^\circ$ is finite, say of order $\ell\in \N$ (where $X^\circ$ denotes the identity component of $X$), we have $\ell x \in X^\circ$. Choosing $y\in X^\circ$ with $\ell x = \ell y$, we see that $x - y$ admits a representative in $\frac{1}{\ell}\Z^d$. Thus, the connected component of $x$ in $X$ is given by
    \[x + X^\circ = (x-y) + X^\circ =  v_0 + \Span_\R \{v_1,\dots,v_m\} + \Z^d\]
    for some $m\le d$, $v_0\in \frac{1}{\ell}\Z^d$ and $v_1,\dots,v_m\in \Z^d$.
    
    Choose sequences of integers $(r_{i,n})_{n=1}^\infty$, for $i=1,\dots,m$, such that the sequence $(\widetilde{x}_n)\subseteq \R^d$, defined by
    \[\widetilde{x}_n \coloneqq v_0 + \sum_i \frac{r_{i,n}}{n}v_i,\]
    satisfies $x_n\coloneqq \widetilde{x}_n+\Z^d \underset{n\to\infty}{\longrightarrow}x$.
    We claim that $x_n$ is of period $p_n\coloneqq(d!\ell )n$. To see that, first observe that for every integer $0\le k <d$, we have $\frac{p_n}{k!}\widetilde{x}_n\in \Z^d$. Now since $U$ is in Jordan normal form, every entry of $U^{p_n}\widetilde{x}_n$ is of the form
    \[U^{p_n}(\widetilde{x}_n)[i] = \widetilde{x}_n[i] + \binom{p_n}{1}\widetilde{x}_n[i+1] + \cdots + \binom{p_n}{k}\widetilde{x}_n[i+k]\]
    for some $0\le k < d$. Thus, $U^{p_n}(\widetilde{x}_n) - \widetilde{x}_n \in \Z^d$, confirming that $U^{p_n} x_n = x_n$.

    For a general unipotent matrix $U$, let $J$ denote its Jordan normal form, and let $Q$ be an integer matrix such that $U=QJQ^{-1}$ (noting that $Q^{-1}$ may not be an integer matrix). Define $Y = \{y\in \T^d:Qy\in X\}$ and choose some $y\in Y$ such that $Qy = x$ (which is possible since $Q$ induces a surjective endomorphism of $\T^d$). Apply the previous part of the proof to find a sequence $(y_n)$ in $Y$ which converges to $y$ and such that $J^{p_n} y_n = y_n$. Then the sequence $(Qy_n)$ has the required properties.
\end{proof}
Recall that for a dynamical system $(X,T)$, $x\in X$ and $n\in \N$ we denote $\mathfrak{m}_{T,x,n} = \frac{1}{n}\sum_{i=0}^{n-1} \delta_{T^i x}$, where $\delta_x$ is the Dirac measure at $x$ (see the discussion preceding Theorem \ref{thm dense ergodic measures}). 
\begin{proposition}\label{prop strong dpm}
    Let $U\in \GL_d(\Z)$ be a unipotent matrix and $\mu\in \M_{U}(\T^d)$ an ergodic measure.
    Then there exists $c\in \N$ and a sequence $(x_n)_{n=1}^\infty$ in $\T^d$ such that $x_n$ is of period $cn$ 
    and $\m_{U,x_n, cn}\underset{n\to\infty}{\longrightarrow}\mu$ (in the weak-* topology).
\end{proposition}
\begin{proof}
    By Proposition \ref{prop measure classification}, $\supp(\mu) = \bigcup_{i=0}^{m-1} U^i a + H$ for some $a\in \T^d$ and a closed connected subgroup $H\le \T^d$.
    If $H$ is finite then so is $\supp(\mu)$. Then for every $n\in\N$ we may take $x_n$ to be any point in $\supp(\mu)$; by ergodicity, the $U$-invariant measure supported on its orbit is exactly $\mu$.
    
    Otherwise, $a+H$ contains non-torsion elements, and by replacing $a$ if necessary, we may assume that $a$ itself is non-torsion. Thus, we can apply Lemma \ref{lemma periodic in subgroup} to obtain $c\in \N$ and a sequence $(x_n)$ contained in the subgroup $\overline{\langle a\rangle}$ such that $x_n$ is of period $cn$ and $x_n\underset{n\to\infty}{\longrightarrow}a$.
    
    Moreover, since $a+H$ is $U^m$-invariant and $x_n \in \overline{\langle a \rangle}$, every coset $x_n + H$ is $U^m$-invariant as well. Consequently, $\m_{U, x_n, cn}$ is a $U$-invariant measure whose support is contained inside $\bigcup_{i=0}^{m-1}U^i x_n +H$, and combined with the fact that $x_n \to a$, it implies that every accumulation point of the sequence $(\m_{U, x_n, cn})_{n}$ is supported on a subset of $\supp(\mu)$. Now by Proposition \ref{prop unique ergodicity} (and the compactness of $\M_U(\T^d)$), $\m_{U, x_n, cn} \to \mu$.
\end{proof}
\subsection{A substitute for partial specification}
Unipotent toral automorphisms generally do not satisfy partial specification. In the final part of this section, we establish Proposition \ref{prop interval perm}, which provides a useful substitute for this property. Rather than requiring a single periodic orbit to approximate an entire orbit segment, we demand only that smaller segments of the periodic orbit approximate parts of the original orbit, while ensuring that both orbits share a similar distribution in $\T^d$. We begin with a simple preparatory lemma.
\begin{lemma}\label{lemma T^k measure convergence}
    Let $(X,T)$ be a dynamical system. Let $(z_n)_{n=1}^\infty$ be a sequence in $X$ and 
    $(q_n)_{n=1}^\infty$ a sequence in $\N$ such that $q_n \to \infty$. Suppose that $\m_{T,z_n,q_n}\to \mu\in \M_T(X)$.
    Let $K\in \N$ and for every $n$ let $p_n = \lfloor q_n / K \rfloor$.
    If the dynamical system $(\supp(\mu), \restr{T^K}{\supp(\mu)})$ is uniquely ergodic, then $\m_{T^K, z_n, p_n}\to\mu$.
\end{lemma}
\begin{proof}
    By the unique ergodicity, it suffices to show that every accumulation point of $(\m_{T^K, z_n, p_n})$ is supported on a subset of $\supp(\mu)$.
    Let $x\in X\setminus \supp(\mu)$ and choose an open ball $B$ containing $x$ with $\overline{B}\subseteq X\setminus \supp(\mu)$, so that $\mu(\overline{B})=0$.
    Since $2p_nK\ge q_n$ for every sufficiently large $n$, we obtain from the Portmanteau theorem
    \[\m_{T^K, z_n, p_n}(B) \le 2K\cdot\m_{T,z_n,q_n}(B) \to 2K\cdot\mu(B) = 0.\]
    Applying the Portmanteau theorem once again, if $\nu$ is any accumulation point of $(\m_{T^K, z_n, p_n})_n$ then $\nu(B)=0$. Hence, $x\notin \supp(\nu)$, and the proof is complete.
\end{proof}
\begin{proposition}\label{prop interval perm}
    Let $U\in \GL_d(\Z)$ be a unipotent matrix, $\mu\in \M_U(\T^d)$ an ergodic measure whose support consists of $m\in \N$ connected components and $x\in \T^d$ a $\mu$-generic point. Then for every $K\in\N$ coprime to $m$ and $\varepsilon>0$ there exists $z\in\T^d$ of period $q>K$ and a permutation $\pi$ of $\{0,\dots,q-1\}\cap K\N_0$ such that
    \[{\Big\lvert}\{0\le i < \frac{q}{K}: d_{\T^d}(U^{iK}z, U^{\pi(iK)}x)<\varepsilon\}{\Big\rvert} > (1-\varepsilon)\frac{q}{K}.\]
    Moreover, there exists $c\in \N$, depending only on $U$ and $\mu$, such that $z$ can be chosen with period $q$ equal to any sufficiently large element of $c\N$.
\end{proposition}
\begin{proof}
    By Proposition \ref{prop measure classification}, $\supp(\mu)= \bigcup_{i=0}^{m-1} U^i a + H$, where $a\in\T^d$, $H\le \T^d$ is a closed connected proper subgroup and the restriction of $\mu$ to each coset is the translate of the Haar measure on $H$. 
    Take an open neighborhood $\supp(\mu)\subseteq O\subseteq \T^d$ sufficiently small so that there exists a partition of boxes $X_1,\dots, X_{r}$ of $O$ that satisfies the following:
    \begin{enumerate}
        \item the diameter of each $X_j$ is less than $\varepsilon$.
        \item No side of $X_j$ is contained in $\supp(\mu)$. \label{item measure 0 boundary}
    \end{enumerate}
    By item \eqref{item measure 0 boundary}, the intersection of $\supp(\mu)$ with any side of $X_j$ is lower dimensional and hence $\mu(\partial X_j) = 0$.
    Thus, the Portmanteau theorem implies that if $\mu_n\to\mu$ then $\mu_n(X_j)\to\mu(X_j)$ and $\mu_n(O) \to 1$.
    
    By Proposition \ref{prop strong dpm} we can choose $c\in \N$ (depending only on $U$ and $\mu$) and a sequence $(z_n)$ in $\T^d$ with $z_n$ of period $cn$ such that $\m_{U,z_n,cn}\to\mu$.
    The system $(\supp(\mu), \restr{U^K}{\supp(\mu)})$ is uniquely ergodic by Proposition \ref{prop unique ergodicity}, and hence we can apply Lemma \ref{lemma T^k measure convergence} to conclude that $\m_{U^K,z_n,\lfloor cn/K\rfloor}\to \mu$ as well.
    Therefore, fixing a sufficiently large $N$ and putting $z\coloneqq z_N$, $q\coloneqq cN$ and $p\coloneqq\lfloor q/K\rfloor$, we have for every $j=1,\dots,r$,
    \begin{equation*}
        {\big\lvert} \m_{U^K,z,p}(X_j) - \mu(X_j) {\big\rvert} < \frac{\varepsilon}{8r}.
    \end{equation*}
    Since $x$ is $\mu$-generic, a similar argument implies that $\m_{U^K,x,\lfloor cn/K\rfloor}\to \mu$, and enlarging $N$ if necessary, we have for every $j=1,\dots,r$,
    \begin{equation*}
        {\big\lvert} \m_{U^K,x,p}(X_j) - \mu(X_j) {\big\rvert} < \frac{\varepsilon}{8r}.
    \end{equation*}
    It follows from the last two equations that for every $j=1,\dots,r$,
    \begin{equation*}
        {\Big\lvert} \lvert\{0\le i < p: (U^K)^i x \in X_j\}\rvert - \lvert\{0\le i < p: (U^K)^i z \in X_j\}\rvert{\Big\rvert} < \frac{\varepsilon p}{4r} .
    \end{equation*}
    Moreover, by enlarging $N$ if necessary, we may assume that $\m_{U^K,x,p}(O) > 1-\varepsilon/4$.
    Thus, we can choose a permutation $\pi$ of $\{0,\dots,q-1\}\cap K\N_0$ such that $U^{iK} z$ and $U^{\pi(iK)}x$ belong to the same $X_j$ for more than
    \[(1-\frac{\varepsilon}{4})p - r\frac{\varepsilon p}{4r} = (1-\frac{\varepsilon}{2})p\]
    values of $i$, and hence for every such $i$, $d_{\T^d}(U^{iK} z,U^{\pi(iK)}x) < \varepsilon$. Enlarging $N$ once again if necessary, we have $(1-\varepsilon/2)p > (1-\varepsilon)q/K$.
\end{proof}
\section{Dense periodic measures}\label{section dpm}
In this section we complete the proof of Theorem \ref{thm dpm intro}.
Denote by $\M_T^{\fs}(X)$ the subset of $\M_T(X)$ consisting  of finitely supported measures. First, we recall the definition of dense periodic measures.
\begin{definition}
    A dynamical system $(X,T)$ has \emph{dense periodic measures} if $\M_T^{\fs}(X)$ is dense in $\M_T(X)$ (with respect to the weak-* topology).
\end{definition}
\begin{remark}\label{remark dpm}
    Since the collection of convex combinations of ergodic measures is dense in $\M_T(X)$, having dense periodic measures is equivalent to the seemingly weaker property that the closure of $\M_T^{\fs}(X)$ contains all of the ergodic measures.

    In addition, if $(X,T)$ has dense periodic measures, then every ergodic measure in $\M_T(X)$ is actually a weak-* limit of a sequence of \textit{ergodic} measures in $\M_T^{\fs}(X)$, see \cite[Proposition 1.5]{phelps} and also \cite[Section 2]{levit_lubotzky}.
\end{remark}
Let $X$ be a compact metric space and  $B\subset C(X)$ the unit ball in the uniform norm. Choose a collection of continuous functions $\{f_n\}_{n=1}^\infty\subset B$ which is dense in $B$, and define a metric on $\mathcal{M}(X)$ (the space of Borel probability measures on $X$) by
\[d_{\mathcal{M}(X)}(\mu,\nu) = \sum_{n=1}^\infty 2^{-(n+1)} {\Big\lvert} \int f_n \;d\mu - \int f_n \;d\nu{\Big\rvert}.\]
It is well known that $d_{\mathcal{M}(X)}$ induces the weak-* topology on $\mathcal{M}(X)$.
\begin{lemma}\label{lemma point measure distance bound}
    Let $X$ be a compact metric space. For every $\varepsilon>0$ there exists $\delta > 0$ with the following property:
    for every $q\in \N$ and sequences $(x_n)_{n=0}^{q-1}, (y_n)_{n=0}^{q-1} $ in $X$,
    \begin{equation*}
         d_{\mathcal{M}(X)}(\frac{1}{q}\sum_{n=0}^{q-1} \delta_{x_n},\frac{1}{q}\sum_{n=0}^{q-1} \delta_{y_n}) \le \varepsilon + \frac{1}{q}{\big\lvert} \{0\le n < q : d_X(x_n, y_n) \ge \delta\}{\big\rvert}.
    \end{equation*}
\end{lemma}
\begin{proof}
    Given $\varepsilon > 0$, choose $\delta>0$ such that
    \[ d_{X}(x,y) < \delta \Longrightarrow d_{\mathcal{M}(X)}(\delta_{x},\delta_{y}) < \varepsilon.\]
    Observe that $d_{\mathcal{M}(X)}$ is a convex function on $\mathcal{M}(X) \times \mathcal{M}(X)$, and hence
    \[d_{\mathcal{M}(X)} (\frac{1}{q}\sum_{n=0}^{q-1} \delta_{x_n},\frac{1}{q}\sum_{n=0}^{q-1} \delta_{y_n}) = d_{\mathcal{M}(X)}{\big(}\frac{1}{q}\sum_{n=0}^{q-1}(\delta_{x_n}, \delta_{y_n}){\big)} \le
    \frac{1}{q}\sum_{n=0}^{q-1} d_{\mathcal{M}(X)} (\delta_{x_n}, \delta_{y_n}).\]
    Whenever $d_{X}(x_n,y_n) < \delta$, the corresponding term on the right-hand side is bounded above by $\varepsilon$; for the remaining terms, we note that $d_{\mathcal{M}(X)} \le 1$, and the lemma follows. 
\end{proof}
The following proposition provides the final key step in the proof of Theorem \ref{thm dpm intro}.
\begin{proposition}\label{prop spec and dpm}
    Let $(X,T)$ be an abelian group dynamical system. Suppose that $Y \le X$ is a $T$-invariant closed subgroup, such that $(Y,T)$ satisfies partial specification and $(X/Y,T)$ is isomorphic to a torus with a unipotent automorphism. Then $(X,T)$ has dense periodic measures.
\end{proposition}
As noted in the introduction, the following is a simple but previously unknown consequence of Proposition \ref{prop spec and dpm}.
\begin{corollary}
    Let $A$ be an automorphism of $\T^d$, induced by a block-diagonal matrix consisting of a unipotent block and a block whose eigenvalues are not roots of unity. Then the system $(\T^d,A)$ admits dense periodic measures.
\end{corollary}
Notably, it seems to us that the current approach cannot be significantly simplified, even for the proof of this special case.

Before presenting the proof of Proposition \ref{prop spec and dpm}, we outline its main steps.
\begin{enumerate}
        \item We aim to approximate the measure $\m_{T,x,q}$ (see the paragraph preceding Theorem \ref{thm dense ergodic measures} for notation), where $q$ is a large integer and $x\in X$ is a generic point for some ergodic measure .
        Using Proposition \ref{prop interval perm} for $X/Y$, we construct a periodic point $z\in X$ such that $\m_{T,z+Y,q}$ approximates $\m_{T,x+Y,q}$,
        and such that long orbit segments of $z+Y$ of length $K$ trace segments of the orbit of $x+Y$.
        \item To lift the approximation from $X/Y$ to $X$, we introduce elements $y_i\in Y$ to account for the discrepancy between $T^{iK} z$ and the corresponding points in the orbit of $x$. Then, using the partial specification in $Y$, we construct a single periodic point $w\in Y$ that approximates the orbit segments of the $y_i$.

        The key point here is that both $z$ and $w$ have period $q$; see the beginning of Subsection \ref{subsec dpm unipotent} for further discussion.
        \item We finally show that the measure $\m_{T,z-w,q}$ approximates $\m_{T,x,q}$.
    \end{enumerate}
\begin{proof}[Proof of Proposition \ref{prop spec and dpm}]
    If $Y = X$, then the density of periodic measures follows from the partial specification (as in the proof of Theorem \ref{thm dense ergodic measures}). Otherwise, $Y$ is a proper subgroup, and
    we may omit the isomorphism and assume that $X/Y = \T^d$ for some $d\in \N$, and that $T$ induces a unipotent automorphism on $X/Y$.
    Let $\lambda\in \M_T(X)$, which we may assume to be ergodic by Remark \ref{remark dpm}. Let $\mu$ denote its projection on $X/Y$; then $\mu$ is also ergodic.
    Let $0<\varepsilon<1$, and let $0<\delta <\varepsilon$ be provided by Lemma \ref{lemma point measure distance bound} (for the space $X$).
    Let $c$ be the constant from Proposition \ref{prop interval perm} corresponding to $T$ and $\mu$.
    Choose $\PP\subseteq c\N$ with bounded gaps, such that $(Y,T)$ satisfies partial specification with periods $\PP$, and let $N,M\in \N$ denote the parameters given by the partial specification for $\delta/2$ (Definition \ref{def partial specification}).
    
    Fix an integer $K>\lceil M/\varepsilon\rceil$ coprime to the number of connected components of $\supp(\mu)$ (which is finite by Proposition \ref{prop measure classification}). Choose $0 < \eta < \varepsilon$ such that for every $k=0,\dots,K-1$ and $x_1,x_2\in X$,
    \begin{equation}\label{eq new new eq eta delta}
        d_X(x_1,x_2)<\eta \Longrightarrow d_{X}(T^k x_1,T^k x_2) < \frac{\delta}{2}.
    \end{equation}
    Let $x\in X$ be a $\lambda$-generic point. Then $x+Y$ is $\mu$-generic. Apply Proposition \ref{prop interval perm} with $T,\mu,x+Y,K,\eta$ to obtain $z\in X$, an arbitrarily large $q\in c\N$ and a permutation $\pi$ of $\{0,\dots,q-1\}\cap K\N_0$ such that $z+Y$ is of period $q$ and $\lvert\Lambda\rvert > (1-\eta)q/K$, where
    \[\Lambda\coloneqq \{0\le i< \frac{q}{K}:d_{X/Y}(T^{iK} z+Y, T^{\pi(iK)}x+Y) < \eta\}.\]
    Invoking Lemma \ref{lemma periodic point in coset} and replacing $z$ if necessary, we may assume that $T^{q}z=z$ (see Remark \ref{remarks partial spec}\eqref{remark partial spec ergodic}).
    For every $i\in \Lambda$ there exists $y_i\in Y$ such that
    \begin{equation}
        d_X(T^{iK} z - T^{\pi(iK)}x, y_i)<\eta,
    \end{equation}
    and Equation \eqref{eq new new eq eta delta} implies that for $k = 0,\dots,K-1$ we have
    \begin{equation}\label{eq z,x distance}
        d_{X}(T^{iK+k} z - T^{\pi(iK)+k}x, T^k y_i) < \frac{\delta}{2}.
    \end{equation}
    For $i\notin \Lambda$, choose $y_i\in Y$ arbitrarily.
    Recall that $N,M\in\N$ are the partial specification parameters defined in the beginning of the proof. Since $q$ can be any arbitrarily large number in $c\N$, by enlarging $q$ if necessary we may assume that $q\ge N$ and that $q\in \PP\subseteq c\N$.
    Define $r\in \N$ and $a_0 < b_0 <\cdots < a_{r} < b_{r}$ as follows: $a_i = iK$ and $b_i = (i+1)K - M$. We choose $r$ to be the maximal integer such that $(1+\frac{\delta}{2})b_{r} \le q$.
    Consider the $M$-spaced specification in $Y$ defined by
    \begin{equation*}
        \{(T^{-a_i} y_i; a_i, b_i)\}_{i=0}^{r},
    \end{equation*}
    then there exists $w\in Y$ of period $q$ that $\delta/2$-partially traces this specification.
    Write $\Lambda_i\subseteq [a_i,b_i)$ for the collection of tracing indices,
    so that for every $i=0,\dots,r$ and $n\in \Lambda_i$,
    \begin{equation}\label{eq y_i w distance}
        d_X(T^{n-a_i} y_i, T^n w) < \frac{\delta}{2}.
    \end{equation}
    Then $z-w$ is of period $q$. We aim to show that $\m_{T, z-w, q}$ approximates $\m_{T, x, q}$.
    
    Observe that for $i\in \Lambda\cap\{0,\dots,r\}$ and $n\in \Lambda_i$,
    \begin{align}\label{eq distance z-w,x}
        & d_X(T^n(z-w), T^{n-a_i+\pi(a_i)}x) = d_X(T^n z - T^{n-a_i+\pi(a_i)}x, T^n w) \le \\
        & d_X(T^n z - T^{n-a_i+\pi(a_i)}x, T^{n-a_i}y_i) + d_X(T^{n-a_i}y_i, T^n w) < \delta, \nonumber
    \end{align}
    where the last inequality was obtained by applying Equations \eqref{eq z,x distance} (with $k = n - a_i$) and \eqref{eq y_i w distance}.
    Denote by $i_{\max}$ the unique integer such that $\pi(i_{\max} K)$ is maximal. In particular, for every $i\ne i_{\max}$, $\pi(iK) < q- K$.
    Let $\tilde{\pi}$ be a permutation of $\{0,\dots,q-1\}$ that extends $\pi$ and satisfies $\tilde{\pi}(a_i + k) = \pi(a_i) + k$ for every $i=0,\dots,r$ with $i\ne i_{\max}$ and $k=0,\dots, K-1$. Then
    \begin{equation*}
        d_{\M_T(X)} (\m_{T, z-w, q}, \m_{T, x, q}) = d_{\M_T(X)} (\m_{T, z-w, q}, \frac{1}{q}\sum_{n=0}^{q-1}\delta_{T^{\tilde{\pi}(n)}x}).
    \end{equation*}
    Denote $\Lambda' = \Lambda\cap\{0,\dots,r\}\setminus\{i_{\max}\}$. Applying Lemma \ref{lemma point measure distance bound} and then Equation \eqref{eq distance z-w,x} we obtain that the above distance is bounded above by
    \begin{equation}\label{eq upper bound z-w}
        \varepsilon + \frac{1}{q}{\Big\lvert} \{0\le n < q : d_{X}(T^n (z-w), T^{\tilde{\pi}(n)}x) \ge \delta\}{\Big\rvert} \le \varepsilon + \frac{1}{q}{\Big\lvert} \{0,\dots,q-1\}\setminus \bigcup_{i\in \Lambda'}\Lambda_i {\Big\rvert},
    \end{equation}
    so we just need to estimate the size of the last set. Note that
    \begin{equation*}
        {\Big\lvert} \bigcup_{i\in \Lambda'}\Lambda_i {\Big\rvert} \ge  (1-\varepsilon)\sum_{i\in \Lambda'}(b_i - a_i) = (1-\varepsilon)(K-M){\big\lvert} \Lambda' {\big\rvert}.
    \end{equation*}
    Recalling that $M<\varepsilon K$ and
    that the complement of $\Lambda$ in the set $\{i\in \Z: 0\le i < \frac{q}{K}\}$ (which contains $\{0,\dots,r\}$) has size less than $\frac{q}{K}\eta$, we obtain
    \[(K-M){\big\lvert} \Lambda' {\big\rvert} > (1-\varepsilon)K (r-\frac{q}{K}\eta) > (1-\varepsilon) (Kr - q\varepsilon).\]
    Note that the choice of $r$ implies that $(1+\frac{\delta}{2})(b_{r} + K) > q$ and therefore by the choice of $b_r$ that $(1+\varepsilon)Kr > q - 2(1+\varepsilon)K$. Enlarging $q$ once again if necessary, we may assume that $K/q < \varepsilon$, and hence
    \[Kr > \frac{q}{1+\varepsilon} - 2K > (\frac{1}{1+\varepsilon}-2\varepsilon)q.\]
    Combining the last three equations, we see that
    \[{\Big\lvert} \bigcup_{i\in \Lambda'}\Lambda_i {\Big\rvert} \ge (1-\varepsilon)^2(\frac{1}{1+\varepsilon}-3\varepsilon)q.\]
    Plugging the last equation into Equation \eqref{eq upper bound z-w}, it follows that
    \[d_{M_T(X)} (\m_{T, z-w, q}, \m_{T, x, q}) \le \varepsilon + 1 - (1-\varepsilon)^2(\frac{1}{1+\varepsilon}-3\varepsilon),\]
    which is arbitrarily small. Since, by enlarging $q$ if necessary, we may assume that $\m_{T, x, q}$ is arbitrarily close to $\lambda$, we conclude that $\lambda\in \overline{\M_T^{\fs}(X)}$.
    \end{proof}
We need two lemmas before proving Theorem \ref{thm dpm intro}.
\begin{lemma}[cf.\ {\cite[Lemma 8.2]{levit_vigdorovich}}]\label{lemma finite index dpm}
    Let $(X,T)$ be a dynamical system and $k\in \N$. If $(X,T^k)$ has dense periodic measures then so does $(X,T)$.
\end{lemma}
\begin{proof}
    Let $\mu\in \M_T(X)$. Then $\mu\in \M_{T^k}(X)$ as well, so there exists a sequence $(\mu_n)$ in $\M_{T^k}^{\fs}(X)$ such that $\mu_n \to \mu$.
    Then $\frac{1}{k}\sum_{j=0}^{k-1} T^j_*\mu_n \in \M_T^{\fs}(X)$ and converges to $\mu$ as $n\to\infty$.
\end{proof}
Although the next lemma appears implicitly in the literature (cf.\ \cite{lawton}), we choose to state it explicitly here.
\begin{lemma}\label{lemma dcc and fg}
    Let $(X,T)$ be an abelian group dynamical system, and let $H= \Z\ltimes \widehat{X}$, where the action of $1\in \Z$ on $\widehat{X}$ is given by $\widehat{T}$. Then $(X,T)$ satisfies the dcc if and only if $H$ is finitely generated.

    Consequently, if $(X,T)$ satisfies the dcc, then so does $(X,T^n)$ for any $n\in \N$.
\end{lemma}
\begin{proof}
    Suppose first that $H$ is finitely generated. Let
    $X\ge X_1\ge X_2\ge\cdots$ be a chain of $T$-invariant closed subgroups, then their annihilators $X_1^\bot\le X_2^\bot\le\cdots\le \widehat{X}$ form a non-decreasing chain of $\widehat{T}$-invariant subgroups. Conjugating elements of $\widehat{X}$ by $n\in \Z$ corresponds to applying $\widehat{T}^n$, and it follows that all $X_i^\bot$ are normal subgroups of $H$. By a theorem of Hall \cite[Theorem 3]{hall} (see also \cite[15.3.1]{a_course_groups}), every non-decreasing chain of normal subgroups of $H$ stabilizes, implying that $(X,T)$ satisfies the dcc.
    
    Conversely, suppose that $H$ is not finitely generated. Then there exists an infinite chain of subgroups $\Z\ltimes G_1 \lneq \Z\ltimes G_2\lneq\cdots\lneq H$, and it follows that $X\gneq G_1^\bot \gneq G_2^\bot \gneq\cdots$ is an infinite chain of $T$-invariant closed subgroups of $X$.

    The final part of the lemma follows from the fact that $(X,T^n)$ corresponds to the finite index subgroup $n\Z\ltimes \widehat{X}$ of $H$, and is therefore finitely generated whenever $H$ is.
\end{proof}
The part of Theorem \ref{thm dpm intro} dealing with ergodic measures has already been proved (Theorem \ref{thm dense ergodic measures}). The following result will therefore complete the proof of the theorem.
\begin{theorem}
    Let $(X,T)$ be an abelian group dynamical system satisfying the dcc. Then $(X,T)$ has dense periodic measures.
\end{theorem}
\begin{proof}
    Let $X_2\le X_1 \le X$ be a chain of $T$-invariant closed subgroups as provided by Proposition \ref{prop nonergodic structure}.
    It is known that abelian group dynamical systems satisfying the dcc admit dense periodic \textit{points} \cite{laxton, kitchens}, \cite[Theorem 5.7]{schmidt_book}. Thus, we can choose periodic points $x_1,\dots,x_k\in X$ such that $\bigcup_{i=1}^k x_i +X_1 = X$.
    Choose $n\in\N$ that is a common period for $x_1,\dots,x_k$ and such that $T^n$ induces a unipotent automorphism on $X_1/X_2$. By Lemma \ref{lemma finite index dpm} it suffices to show that $(X,T^n)$ has dense periodic measures. As $T$ is an automorphism of a compact abelian group, the ergodicity of $(X_2,T)$ implies the ergodicity of $(X_2,T^n)$ (see e.g.\ \cite[Theorem 1.6]{schmidt_book}), and the latter satisfies the dcc by Lemma \ref{lemma dcc and fg}.
    By Theorem \ref{thm spec intro}, $(X_2, T^n)$ satisfies partial specification, and then Proposition \ref{prop spec and dpm} implies that $(X_1,T^n)$ has dense periodic measures. Since $T^n x_i = x_i$ for each $i=1,\dots,k$, the map $x\mapsto x_i + x$ conjugates the systems $(X_1, T^n)$ and $(x_i + X_1, T^n)$, and therefore the latter has dense periodic measures as well.
    This concludes the proof since any ergodic measure in $\M_{T^n}(X)$ is supported on a single coset $x_i + X_1$ (see Remark \ref{remark dpm}).
\end{proof}
\section{Proofs of the other main results}\label{section other main proofs}
In this section, we prove the corollaries of Theorem \ref{thm dpm intro} that were stated in the introduction.
\begin{proof}[Proof of Corollary \ref{cor hs}]
    Let $\Z\ltimes G$ be finitely generated with $G$ a countable abelian group, and denote by $\phi$ the automorphism of $G$ corresponding to the action of $1\in \Z$.
    It follows from Lemma \ref{lemma dcc and fg} that the abelian group dynamical system $(\widehat{G},\widehat{\phi})$ satisfies the dcc. By Theorem \ref{thm dpm intro}, $(\widehat{G},\widehat{\phi})$ has dense periodic measures, and \cite[Proposition 10.2]{levit_vigdorovich} or \cite[Corollary 5.19]{eckhardt_shulman_amenable} assert that this is equivalent to $\Z\ltimes G$ being Hilbert-Schmidt stable.
\end{proof}
\begin{proof}[Proof of Corollary \ref{cor livshitz}]
    By \cite[Proposition 10.13]{katok_construct}, the following equality holds (in every compact metric space):
    \[\cl_{C(X)}\{P\circ T - P : P\in C(X)\} = \{f\in C(X):\int_X f\;d\mu = 0\text{ for all }\mu\in \M_T(X)\}.\]
    Since $(X,T)$ admits dense periodic measures (Theorem \ref{thm dpm intro}), the right-hand side is equal to
    \[\{f\in C(X):\int_X f\;d\mu = 0\text{ for all }\mu\in \M_T^{\fs}(X)\},\]
    which precisely corresponds to satisfying Equation \eqref{eq periodic sum} for every periodic point.
\end{proof}
\section{Products of systems with dense periodic measures}\label{section product}
This section is independent of the preceding ones and is devoted to proving Theorem \ref{thm product intro}, using only the definitions and notation introduced earlier.

We will say that a dynamical system $(Y,T)$ is a \emph{common subsystem} of the systems $(X_1,T_1)$ and $(X_2,T_2)$ if for $i=1,2$ there is a $T_i$-invariant closed subset $Y_i\subseteq X_i$ such that $(Y_i,T_i)$ is conjugate to $(Y,T)$. It will cause no confusion to omit the conjugating maps and treat $Y$ as a subset of $X_1$ and $X_2$ and $T$ as $\restr{T_i}{Y}$.
\begin{lemma}\label{lemma product no dpm}
    Let $(X_1,T_1)$ and $(X_2,T_2)$ be two dynamical systems that admit a common subsystem $(Y,T)$, such that there exists an ergodic measure $\mu\in \M_{T}(Y)$ which is supported on at least two points.
    Suppose in addition that for every pair of periodic points $x_1 \in X_1$ and $x_2\in X_2$, the least periods of $x_1$ and $x_2$ are coprime. Then $(X_1 \times X_2, T_1\times T_2)$ does not admit dense periodic measures.
\end{lemma}
\begin{proof}
    Let $\mu$ be the measure mentioned above and define $\nu$ to be the diagonal embedding of $\mu$ in $X_1\times X_2$, namely, $\nu = \Delta_*\mu$ where $\Delta:Y \to X_1 \times X_2$ is defined by $x\mapsto (x,x)$.
    Suppose for contradiction that $(\nu_n)$ is a sequence of finitely supported measures in $\M_{T_1\times T_2}(X_1\times X_2)$ converging to $\nu$. Since $\nu$ is ergodic, Remark \ref{remark dpm} allows us to assume that every $\nu_n$ is ergodic as well, which means that it is supported on a single periodic orbit of some point $(x_1, x_2)$. Since the least periods of $x_1$ and $x_2$, denoted $q_1$ and $q_2$, are coprime, the least period of $(x_1,x_2)$ is $q_1q_2$. Consequently,
    \[\nu_n = \frac{1}{q_1q_2}\sum_{m=0}^{q_1-1}\sum_{k=0}^{q_2-1}\delta_{T_1^m x_1}\times \delta_{T_2^k x_2} = \m_{T_1,x_1, q_1} \times \m_{T_2,x_2, q_2}\]
    (see the discussion preceding Theorem \ref{thm dense ergodic measures} for the notation).
    However, the convergence of the product measures $(\nu_n)$ to the diagonally embedded measure $\nu$ contradicts Lemma \ref{lemma product diagonal} below.
\end{proof}
\begin{lemma}\label{lemma product diagonal}
    Let $X_1, X_2$ and $Y$ be compact metrizable spaces, with $Y$ embedded in both $X_1$ and $X_2$. Suppose $\nu$ is a Borel probability measure on $X_1\times X_2$ with $\supp(\nu) \subseteq \{(y,y):y\in Y\}$ and $\lvert \supp(\nu) \rvert \ge 2$. Then for any sequences $(\nu_n^i)\subseteq \M(X_i)$, $i=1,2$, it holds that $\nu_n^1 \times \nu_n^2 \centernot\longrightarrow  \nu$ as $n\to\infty$.
\end{lemma}
\begin{proof}
    For $i=1,2$, choose a point $y_i\in Y$ and a subset $U_i\subseteq Y$ containing $y_i$ and open in $X_i$, such that $(y_i,y_i)\in \supp(\nu)$ and $\overline{U_1}\cap \overline{U_2}=\emptyset$ (since $Y$ is compact, the closures coincide in $Y$ and $X_i$).
    Suppose for contradiction that $\nu_n^1\times\nu_n^2\longrightarrow\nu$. Then by the Portmanteau theorem,
    \[\limsup_{n\to\infty} (\nu_n^1\times\nu_n^2)(\overline{U_1}\times\overline{U_2}) \le \nu(\overline{U_1}\times\overline{U_2})=0,\]
    and thus after passing to a subsequence, $\nu_n^i(U_i)\to 0$ for some $i$; without loss of generality, assume $i=1$. Another application of the Portmanteau theorem shows that $\nu(U_1\times U_1)=0$, contradicting the fact that $(y_1,y_1)\in \supp(\nu)$.
\end{proof}
Let $\mathcal{A}$ be a finite set of symbols, and $\mathcal{L}\subseteq \mathcal{A}^{\N_0}$ a collection of infinite words. The (\emph{bi-infinite}) \emph{subshift generated by $\mathcal{L}$} is defined as the set
\[\{x\in \mathcal{A}^\Z: \text{every finite subword of }x\text{ is a subword of some word in }\mathcal{L}\},\]
equipped with the topology induced from the product topology on $\mathcal{A}^\Z$.
Note that this is a shift-invariant compact metric space.
For $x\in \mathcal{A}^\Z$ and $i\in \Z$, we write $x[i]$ for the letter at position $i$, so that $x = \cdots x[-1]\, x[0] \, x[1]\cdots$. For integers $i<j$, we denote $x[i,j) \coloneqq x[i]\cdots x[j-1]$.

Let $u\in \{0,1\}^{\N_0}$ denote the Thue-Morse sequence:
\[u = 0 \,1\,1\,0\,1\,0\,0\,1\cdots,\]
and let $X_u\subseteq \{0,1\}^\Z$ be the subshift generated by $\{u\}$;
see, for example, \cite{fogg, queffelec} for further details on $X_u$.

Let $a$ be an additional symbol, and for each integer $m \ge 2$, define $x_m\in \{0,1,a\}^{\N_0}$ to be the infinite periodic word
\[x_m = u[0,m-1)\, a\,u[0,m-1)\,a\,\cdots,\]
and let $\tilde{x}_{m}\in \{0,1,a\}^{\Z}$ denote the bi-infinite word obtained by extending $x_m$ periodically.
For an integer $p\ge 2$, let $X_p\subseteq \{0,1,a\}^\Z$ be the subshift generated by $
\{x_{p^n}:n\in\N\}$.
We begin by classifying the elements of $X_p$.
\begin{lemma}\label{lemma elements of Xp}
    Let $p\ge2$ be an integer. Every element of $X_p$ is either:
    \begin{enumerate}
        \item a shift of $\tilde{x}_{p^n}$ for some $n$, or \label{item x lemma xp}
        \item equal to $u$ on the non-negative coordinates, or 
        \item an element of $X_u$.
    \end{enumerate}
\begin{proof}
    Let $y\in X_p$. By the definition of a generated subshift, there exists a sequence $(y_n)\subset \{0,1,a\}^\Z$ converging to $y$, where each $y_n$ is a shift of $\tilde{x}_{p^{m_n}}$ for some $m_n\in \N$.
    First suppose that $a$ does not occur in $y$.
    For every $k\in \N$ we have $y[-k,k] = y_n[-k,k]$ for all sufficiently large $n$, and it follows from the construction of $x_{p^{m_n}}$ that $y[-k,k]$ must be a subword of $u$. Therefore in this case $y\in X_u$.
    Now assume $a$ occurs in $y$, and by shifting $y$ we may also assume that $y[-1] = a$. If the sequence $(m_n)$ is bounded, then the sequence $(y_n)$ contains only finitely many distinct elements and hence $y$ is one of them; this corresponds to case \eqref{item x lemma xp}. If the sequence $(m_n)$ is unbounded, then the non-negative coordinates of $(y_n)$ converge to $u$.
\end{proof}
\end{lemma}
The dynamical system in Theorem \ref{thm product intro} is given by the product $X_2\times X_3$. We first show that both factors admit dense periodic measures.
\begin{lemma}\label{Xp has dpm}
    Let $p\ge2$ be an integer and let $T:X_p\to X_p$ be the left shift. Then $(X_p, T)$ admits dense periodic measures.
\end{lemma}
\begin{proof}
    Let $\mu\in \M_T(X_p)$. By Remark \ref{remark dpm} we may assume that $\mu$ is ergodic, so let $y\in X_p$ be a $\mu$-generic point. Let us consider each of the cases described in the previous lemma. If $y$ is a shift of $\tilde{x}_{p^n}$ for some $n$, then $\mu$ is finitely supported, as $\tilde{x}_{p^n}$ is periodic, and there is nothing more to prove.
    If $y$ is equal to $u$ on the non-negative coordinates, then it is clear that the sequence of finitely supported $T$-invariant measures $(\m_{T, \tilde{x}_{p^n},p^n})_n$ also converges to $\mu$.
    Finally, suppose that $y\in X_u$. Since the left shift on $X_u$ is uniquely ergodic (see \cite{fogg,queffelec}), every point in $X_u$ is $\mu$-generic. In particular we can choose one that agrees with $u$ on the non-negative coordinates. The conclusion then follows as in the previous case.
\end{proof}
\begin{proof}[Proof of Theorem \ref{thm product intro}]
    Lemma \ref{Xp has dpm} asserts that $(X_2,T)$ and $(X_3, T)$ admit dense periodic measures.
    By Lemma \ref{lemma elements of Xp}, combined with the fact that $X_u$ contains no periodic points \cite{fogg}, the least periods of the periodic points in $X_2$ and $X_3$ are $\{2^n:n\in \N\}$ and $\{3^n:n\in \N\}$ respectively. Thus, taking $X_u$ as the common subsystem of $X_2$ and $X_3$, Lemma \ref{lemma product no dpm} implies that $(X_2\times X_3, T\times T)$ does not admit dense periodic measures.
\end{proof}

\bigskip

\DeclareEmphSequence{\itshape}
\bibliographystyle{amsalpha}
\bibliography{references}
\end{document}